\documentclass[a4paper,11pt,twoside]{article}

\usepackage{amsmath, amssymb, amsthm, ascmac}
\usepackage{actuarialsymbol}
\usepackage{color}
\newcommand{\red}{\color[rgb]{0,0,0}}

\usepackage[dvipdfmx]{graphicx}

\usepackage{times}

\newtheorem{thm}{Theorem}[section]
\newtheorem{lemma}[thm]{Lemma}
\newtheorem{cor}[thm]{Corollary}
\newtheorem{remark}[thm]{Remark}
\newtheorem{defn}[thm]{Definition}

\numberwithin{equation}{section}
\allowdisplaybreaks

\def\dis{\displaystyle}
\def\R{\mathbb{R}}

\def\N{\mathbb{N}}

\def\e{{\epsilon}}
\def\D{\Delta}
\def\d{\delta}
\def\l{\left}
\def\r{\right}
\def\a{\alpha}
\def\b{\beta}
\def\G{\Gamma}
\def\g{\gamma}
\def\Th{\Theta}
\def\th{\theta}
\def\S{\Sigma}
\def\s{\sigma}

\def\k{\kappa}

\def\vp{\varphi}

\def\p{\partial}
\def\F{\mathcal{F}}

\def\P{{\mathbb{P}}}
\def\toP{\stackrel{p}{\to}}
\def\toD{\stackrel{d}{\to}}
\def\E{\mathbb{E}}
\def\mb#1{\mbox{\boldmath $#1$}} 

\def\I{\mathbf{1}}
\def\wh#1{\widehat{#1}} 
\def\ol#1{\overline{#1}} 
\def\wt#1{\widetilde{#1}}

\def\df{\mathrm{d}}

\newcommand{\bi}{\begin{itemize}}
\newcommand{\ei}{\end{itemize}}

\title{Approximation and estimation of scale functions for spectrally negative L\'evy processes}
\author{Haruka Irie\footnote{{\tt i.h.swim1729@toki.waseda.jp }} \\
{\it Graduate School of Fundamental Science and Engineering, Waseda University}\\
Yasutaka Shimizu\footnote{Corresponding author: {\tt shimizu@waseda.jp}} \\ 
{\it Department of Applied Mathematics, Waseda University} 
}
\date{October 23, 2024}

\begin{document}

\maketitle 

\begin{abstract} 

The scale function holds significant importance within the fluctuation theory of L\'evy processes, particularly in addressing exit problems. However, its definition is established through the Laplace transform, thereby lacking explicit representations in general. This paper introduces a novel series representation for this scale function, employing Laguerre polynomials to construct a uniformly convergent approximate sequence. Additionally, we derive statistical inference based on specific discrete observations, presenting estimators of scale functions that are asymptotically normal.

\begin{flushleft}
 \vspace{1mm} 
\end{flushleft}
\end{abstract}

\section{Introduction}\label{sec:intro}

On a stochastic basis $(\Omega, \F, \P; \mathbb{F})$ where $\mathbb{F}=(\F_t)_{t\ge 0}$ is a (right-continuous) filtration, we consider a {\it spectrally negative L\'evy process} $X=(X_t)_{t\ge 0}$ starting at $x\in \R$:  
\[
X_t = x + ct + \s W_t - L_t,\quad t\ge 0, 
\]
where $x,c\in \R$ are constants, $W=(W_t)_{t\ge 0}$ is an $\mathbb{F}$-Wiener process and $L=(L_t)_{t\ge 0}$ is an $\mathbb{F}$-L\'evy {subordinator} (possibly of infinite activity) with the L\'evy measure $\nu$ on $(0,\infty)$ satisfying that  
\[
\int_0^1 z\,\nu(\df z) < \infty;\quad \nu([1,\infty)) < \infty.
\]
Note that $W_0=L_0=0\ a.s.$. The {\it Laplace exponent} of $X$ is defined as 
\[
\psi_X(\th) := \log \E\l[e^{\th (X_1 - x)}\r] =  c\th + \frac{\s^2}{2} \th^2 +\int_0^\infty (e^{-\th z} -1)\,\nu(\df z),\quad \th \ge 0. 
\]

We are interested in the {\it $q$-scale function} of $X$, $W^{(q)}:\R\to \R_+:=[0,\infty)$ for a number $q\ge 0$, defined as follows: $W^{(q)}(x) = 0$ on $(-\infty,0)$ and otherwise $W^{(q)}$ is the unique continuous function and right continuous at the origin with the Laplace transform 
\[
\int_0^\infty e^{-\th z}W^{(q)}(z)\,\df z = \frac{1}{\psi_X(\th) - q}\qquad \mbox{for \ \ $\th > \Phi(q)$}, 
\] 
where $\Phi(q)$ is called the {\it Lundberg exponent}: 
\[
\Phi(q) := \sup\{\th \ge 0\,|\, \psi_X(\th) = q\}. 
\]

The scale functions play essential roles in the fluctuation theory of L\'evy processes and have various applications in insurance and finance. 
For example, defining two stopping times $\tau_\a^+=\inf\{t>0\,|\, X_t>\a\}$ and $\tau_\a^- = \inf\{t>0\,|\, X_t < \a\}$ for each $\a\in \R$, 
we have a fluctuation identity for $a>0$ such that 
\[
\E\l[e^{-q \tau_a^+}\I_{\{\tau_a^+ < \tau_0^-\}}\r] = \frac{W^{(q)}(x)}{W^{(q)}(a)} \qquad \mbox{for $x \in [0,a]$}, 
\]
which is an essential identity in the theory of two-sided exit problems and {\red useful} in analyzing credit risks and barrier options,  among others. As $q=0$ we have that $\P(\tau_0^- <\tau_a^+) = 1 - W^{(0)}(x)/W^{(0)}(a)$, and so assuming that 
\begin{align*}
\psi'(0+) = c - \int_0^\infty z\,\nu(\df z) >0, 
\end{align*}
called the {\it net profit condition} in ruin theory, which implies that $W^{(0)}(\infty) = 1/\psi'(0+)$, we have the well-known identity for the ruin probability in the classical ruin theory: 
\begin{align}
\P(\tau_0^- < \infty) = 1 - \psi'(0+)W^{(0)}(x). \label{ruin prob}
\end{align}
See Kyprianou \cite{k14} for the details around these identities. 
Moreover, Biffis and Kyprianou \cite{bk10} and Feng and Shimizu \cite{fs13} showed that scale functions are useful tools to represent more general ruin-related risks. They showed some generalized Gerber-Shiu functions, the classical version introduced by Gerber and Shiu \cite{gs98}, are also represented by $q$-scale functions. 
The scale functions also appear in the optimal dividend problems; see Loeffen \cite{lo09}, and the Parisian-ruin probability; see Loeffen {\it et al.} \cite{letal13} and Baurdoux {\it et al.} \cite{betal16}, among others. The background of the above connection between ruin theory and $q$-scale functions includes the potential theory of spectrally negative L\'evy processes and the Wiener-Hopf factorization; see, e.g., Bertoin \cite{b96, b97} and Roger \cite{r90}. For more details, the reference list in Kyprianou \cite{k14}, Kyprianou and Rivero \cite{kr08},  Feng and Shimizu \cite{fs13}, and Kuznetsov {\it et al.} \cite{kk12} may be helpful for the historical background of these and their applications in various fields.

In considering such an application, 
{\red it will be understood that there is a practical need to identify the scale function and to estimate it statistically from observations of a given L\'evy process. Indeed, the identification and approximation of scale functions is a topic that has attracted much attention in recent years. 

We can find an explicit representation for some simple cases, like a compound Poisson processes; see. e.g., Hubalek and Kyprianou \cite{hk04}. However, obtaining a general representation via the Laplace transform is generally challenging because the Laplace inversion is too difficult to implement. Therefore, attempts are made to get an approximate representation. The earliest work on an approximation of scale functions would be due to Egami and Yamazaki \cite{ey14}, who construct an approximate sequence of $q$-scale functions using one of a compound Poisson-type L\'evy process with phase-type jumps, which consist of a dense family in a class of spectrally negative L\'evy processes. Landrault and Willmot \cite{lw20} proposed an asymptotic expansion for Wiener-Poisson risk models by inverting the Laplace transform of scale functions and investigating some examples where explicit expansions are obtained. Behme {\it et al.} \cite{b23} extends their results to more general L\'evy processes with infinite jumps. Moreover, Xie {\it et al.} \cite{xcz24} focused on a specific probabilistic representation of the $q$-scale function and provided an approximation formula using Laguerre series expansion. See also Mart\'in-Gonz\'alez {\it et al.} \cite{m24} for alternative expansion, and Surya \cite{su07} for numerical methods, among others. 

Thus, there are many discussions on the approximation of scale functions, but statistical inference has yet to be discussed, to the best of our knowledge. Our paper's novelty is that it gives a new series approximation of the $q$-scale function and a data-based statistical estimation of the scale function. 
Among these, we are particularly concerned with problems in insurance actuarial practice. In modern actuarial practice, it is standard to use the spectrally negative L\'evy process $X=(X_t)_{t\ge 0}$ for the surplus or asset processes of insurance companies, and certain discrete observations of $X$ are available as real data; see Section \ref{sec:sampling}. However, it is usually unclear from the data which L\'evy process these data follow. We therefore propose a method for estimating scale functions without specifying a specific model of $X$, by using a nonparametric method for estimating quantities associated with the L\'evy measure.}

We must take two steps to identify the $q$-scale function in practice.  First, we introduce a new approximation formula. We concentrate on a compound geometric integral representation of the $q$-scale function obtained by Feng and Shimizu \cite{fs13}. We derive a Laguerre series expansion of the corresponding compound geometric distribution function, the Stieltjes integral with respect to that gives the expansion of the $q$-scale function. 
Although we also use Laguerre expansion as in Xie {\it et al.} \cite{xcz24}, our approach differs from theirs, 
and the formula is fundamentally different from theirs, which makes the primary contribution of this paper. 

Second, we proceed to statistical inference.  The following two studies, Zhang and Su \cite{zs19} and Shimizu and Zhang \cite{sz19},  are instructive.
The former proposes an estimator of Gerber-Shiu functions by deriving its Laguerre series expansion and estimating the coefficients of each term. 
They show a consistency for their proposed estimator. 
Shimizu and Zhang \cite{sz19} applied the same approach to the ruin probability, {\red and} they further showed that the estimator {\red is} asymptotically normal. 
As shown in equation \eqref{ruin prob}, the ruin probability is represented by $W^{(0)}(x)$, so their estimator is also an asymptotically normal estimator for the 0-scale function. This paper will construct an asymptotically normal estimator of the $q$-scale function, a generalization of \cite{sz19}.

The paper is organized as follows. Section \ref{sec:q-scale} introduces the series representation of the $q$-scale function obtained by Feng and Shimizu \cite{fs13}. Under the net profit condition, the $q$-scale function has an integral representation with respect to the compound geometric distribution. 
{\red We then construct an approximation of the compound geometric distribution function using the Laguerre series expansion, which can yield the series expansion of the $q$-scale function. Section \ref{sec:inference} is the main part and gives the estimator for the approximate sequence of $q$-scale functions. 
Under a discrete sampling scheme for the underlying process $X$, we show their consistency and asymptotic normality with the explicit asymptotic variance. 
The auxiliary statistics, their properties, and some proofs are summarised in the Appendix.}

\begin{flushleft}
{\bf \large Notation}
\end{flushleft}
Throughout the paper, we use the following notation:  
\bi
\item $\N = \{1,2,3,\dots\}$; $\N_0:=\N \cup \{0\}$; $\R_+:=[0,\infty)$. 
\item For a $d\times d$ matrix $A=(a_{ij})_{1\le i,j\le d}$, denote by $|A| :=\l( \sum_{i,j=1}^d a_{ij}^2 \r)^{1/2}$. 

\item For $\mb{a}=(a_1,\dots,a_d)^\top$ and $\mb{b}=(b_1,\dots,b_d)^\top$, 
the inner product is given by $\mb{a}\cdot \mb{b} =\sum_{k=1}^d a_kb_k$.

\item Denote by $\mb{0}_d$ the $d$-dim zero vector, and $I_d$ the $d\times d$ identity matrix. 
Moreover,  $N_d(\mb{a}, \S)$ means the $d$-dim Gaussian distribution with mean vector $\mb{a}$ and covariance matrix $\S$. 
In particular, $N:=N_1$. 

\item $\I_A(x)$ is the indicator function on a set $A\subset \R$: $\I_A(x) = 1$ if $x\in A$; 0 otherwise. 

\item For functions $f,g$ on $\R$, denote by $f(x) \lesssim g(x)$ if there exists a constant $C>0$ such that 
$f(x) \le C g(x)$ for any $x\in \R$. 

\item For a function $f(x_1,x_2,\dots,x_d)$, denote by $\p_{x_i} f := \frac{\p f}{\p x_i}$ and $\p_{(x_1,\dots,x_d)} f = (\p_{x_1} f,\dots, \p_{x_d}f)^\top$. 

\item For $s\ge 1$ and $\a>0$, denote by 
\[
L^s_\a(\R_+) = \l\{f:\R_+\to \R\,:\, \int_0^\infty f^s(x)e^{-\a x}\,\df x < \infty \r\}. 
\]
In particular, we denote by $L^s(\R_+) := L^s_0(\R_+)$, and for $f,g\in L^2(\R_+)$, 
\[
\langle f,g\rangle := \int_0^\infty f(x)g(x)\,\df x,\quad \|f\|:= \sqrt{\langle f,f\rangle}. 
\]
\item For functions $f,g$ on $\R_+$ of finite variation, the convolution is defined as 
\[
f*g (x) := \int_{[0,x]}f(x-y)g(y)\,\df y,\quad x\ge 0. 
\]
Moreover, the $k$th convolution is defined by $f^{*0}= \d_0$ (Dirac's delta function concentrated on $0$) and $f^{*k} = f*f^{*(k-1)}$ for $k\in \N$.  

\item For a measure $\nu$ on $\R_+$ and a function $\mb{f}=(f_1,\dots,f_d)^\top:\R_+\to \R^d$, we denote by 
\[
\nu(\mb{f}) :=\int_{\R_+} \mb{f}(x)\,\nu(\df x) = \l(\int_{\R_+} f_1(x)\,\nu(\df x),\dots,\int_{\R_+} f_d(x)\,\nu(\df x)\r)^\top. 
\]
\ei
 
\section{Series representation of $q$-scale functions}\label{sec:q-scale}

\subsection{A compound geometric representation of $q$-scale functions}
A closed expression for the $q$-scale function of the process $X$ is found by Feng and Shimizu \cite{fs13} in an expectation with respect to a compound geometric distribution under the following {\it net profit condition}:  
\bi
\item[{\bf NPC}]: $\dis c > \int_0^\infty z\,\nu(\df z)$.   
\ei

To state the representation of $W^{(q)}$, we prepare some notation: In the sequel, we fix a number $q\ge 0$ for $W^{(q)}$, and let 
\[
D:=\frac{\s^2}{2};\qquad \g:=\Phi(q);\qquad 
\b=
\begin{cases}
cD^{-1} + \g, & (D>0) \\
\g, & (D=0)
\end{cases}; 
\]
\begin{align*}
\wt{f}_q(x) :=\l\{
\begin{aligned}
&D^{-1}\int_0^x \df y \int_y^\infty e^{-\b(x-y)}e^{-\g(z-y)}\,\nu(\df z), & (D>0) \\
&c^{-1}\int_x^\infty e^{-\g(z-x)}\,\nu(\df z), &  (D=0) 
\end{aligned}\r..
\end{align*} 
We define a distribution $F_q$ with a probability density obtained from the function $\wt{f}_q$: 
\[
F_q(x) := \int_0^x f_q(z)\,\df z,\quad f_q(z) = p^{-1} \wt{f}_q(z),\quad p := \int_0^\infty \wt{f}_q(z)\,\df z.  
\]
In fact, $p\in (0,1)$ under NPC and the probability density $f_q$ is well-defined; see Lemma \ref{lem:fq}. 
Then, the following representation is the immediate consequence from Feng and Shimizu \cite{fs13}, Proposition 4.1. 

\begin{lemma}\label{lem:fs11}
Suppose NPC for the process $X$. 
Then the $q$-scale function $W^{(q)}$ of $X$ has the following integral form: 
{\small 
\[ 
W^{(q)}(x) = \l\{
\begin{aligned}
&\frac{e^{\g x} - e^{-\b x}}{D(1-p)(\b+ \g)} - \frac{1}{D(1-p)(\b+ \g)}\int_{[0,x)} \l[\g e^{\g(x-z)} + \b e^{-\b(x-z)}\r] \ol{G}_q(z)\,\df z, & (D>0) \\
&\frac{e^{\g x}}{c(1-p)} - \frac{1}{c(1-p)}\int_{[0,x)} e^{\g(x-z)}\ol{G}_q(z)\,\df z,& (D=0)
\end{aligned}\r., 
\]
}
where $p\in (0,1)$ and $G_q$ is a compound geometric distribution function 
\begin{align}
G_q(x) = \sum_{k=0}^\infty (1-p)p^k \int_0^x f_q^{*k}(z)\,\df z,\quad x\ge 0, \label{Gq}
\end{align}
and $G_q(x)\equiv 0$ for $x<0$. 
\end{lemma}

\begin{proof}
See Section \ref{proof:fs11}. 
\end{proof}

\subsection{Laguerre series expansion}
Let $L_k(x)$ be the {\it Laguerre polynomial} of order $k$: 
\[
L_k(x) = \frac{1}{n!}e^x \frac{\df^n}{\df x^n} (e^{-x}x^n) = \sum_{j=0}^k {k\choose j} \frac{(-x)^j}{j!}, \quad x\in \R_+:=[0,\infty), 
\]
by which we define the {\it Laguerre function} for each $\a>0$ as follows: 
\[
\vp_{\a,k}(x) := \sqrt{2\a}L_k(2\a x) e^{-\a x},\quad x\in \R_+. 
\] 
Note that the system $\{\vp_{\a,k}\}_{k\in \N_0}$ consists of an orthogonal basis of $L^2(\R_+)$ 
satisfying that 
\begin{align}
\sup_{k\in \N_0,x\in \R_+}|\vp_{\a,k}(x)|\le \sqrt{2\a }, \label{vp-bdd}
\end{align}
and that, for any $k\in \N_0$, there exists some $\d>0$ such that 
\[
\vp_{\a,k}(x) = O(e^{-(\a - \d)x}),\quad x\to \infty,  
\]
since $L_k$ is a polynomial of order $k$. Therefore, we can control the decay order of the function $\vp_{k,\a}$ by choosing $\a>0${\red ; see Remark \ref{rem:alpha}. }

Let ${\cal W}_\a^r(\R_+)$ be the {\it Sobolev-Laguerre space} (Bongioanni and Torrea \cite{bt09}) for $\a, r>0$: 
\[
{\cal W}_\a^r(\R_+) := \l\{ f\in {\red L^2(\R_+)} \,:\, \sum_{k=0}^\infty k^r |\langle f,\vp_{\a,k}\rangle|^2 < \infty\r\} 
\]
There are many useful links between the Laguerre systems and the Sobolev-Laguerre space in Comte and Genon-Catalot \cite{cg15}, where Proposition 7.1 and its remark provide an equivalent condition that a function belongs to ${\cal W}_\a^r(\R_+)$ as follows. 

\begin{lemma}[Comte and Genon-Catalot \cite{cg15}]\label{lem:SL}
Let $r\in \N$ and $\a>0$. Then, $f\in {\cal W}_\a^r(\R_+)$ if and only if $f$ admits the derivatives of order $r$ and that  
\begin{align}
x^{m/2} \p_x^m f(x)   \in L^2_\a(\R_+),\quad 0\le m\le r. \label{SL-cond}
\end{align}
\end{lemma}

The following uniform convergence of the Laguerre series expansion is obtained by Shimizu and Zhang \cite{sz19}, Proposition 3. 
\begin{lemma}\label{lem:sz19}
For any $K\in \N$, the partial sum of the Laguerre series expansion of $f\in {\cal W}_\a^r(\R_+)$: 
\[
f_K(x):= \sum_{k=0}^K \langle f,\vp_{\a,k}\rangle \vp_{\a,k}(x),\quad x\ge 0, 
\]
satisfies that 
\[
\sup_{x\in \R_+} |f_K(x) - f(x)|  = O\l(K^{-(r-1)/2}\r),\quad K\to \infty. 
\]
\end{lemma}

\subsection{Laguerre expansion of $\ol{G}_q$}

According to, e.g., Willmot and Lin \cite{wl01}, we see that the tail function $\ol{G}_q:=1 - G_q$ 
should satisfy the following defective renewal equation (DRE): 
 \begin{align}
 \ol{G}_q(x) =  p\ol{F}_q +  pf_q(x)* \ol{G}_q(x),\quad x\ge 0.  \label{DRE-Gq}
 \end{align}
We have the Laguerre expansion of $pf_q$, $p\ol{F}_q$ and $\ol{G}_q$ since these belong to $L^2(\R_+)$ as shown in Section \ref{sec:fFG}: for $x\ge 0$, 
\begin{align}
pf_{q}(x)=\sum_{k=0}^{\infty}a^{f}_{\a,k}\vp_{\a,k}(x);\quad
p\ol{F}_{q}(x)=\sum_{k=0}^{\infty}a^{F}_{\a,k}\vp_{\a,k}(x);\quad
\ol{G}_{q}(x)=\sum_{k=0}^{\infty}a^{G}_{\a,k}\vp_{\a,k}(x), \label{L-exps}
\end{align}
where $a^{f}_{\a,k}:=\langle pf_{q}, \vp_{\a,k}\rangle$, 
$a^{F}_{\a,k}:=\langle p\overline{F}_{q}, \vp_{\a,k}\rangle$ and 
$a^{G}_{\a,k}:=\langle \overline{G}_{q}, \vp_{\a,k}\rangle$. 

For arbitrary $K\in \N_0$, letting 
\begin{align*}
\mb{a}^{f}_{\a,K}&:=(a^{f}_{\a,0},a^{f}_{\a,1},...,a^{f}_{\a,K})^{\top}; \\
\mb{a}^{F}_{\a,K}&:=(a^{F}_{\a,0},a^{F}_{\a,1},...,a^{F}_{\a,K})^{\top}; \\
\mb{a}^{G}_{\a,K}&:=(a^{G}_{\a,0},a^{G}_{\a,1},...,a^{G}_{\a,K})^{\top}, 
\end{align*} 
we have the following relation among these coefficients; see Zhang and Su \cite{zs18}, Section 2, or Shimizu and Zhang \cite{sz19}, Proposition 2. 

\begin{lemma}\label{lem:zs}
Define a $(K+1)\times(K+1)$-matrix $A^f_{K}=(a_{kl})_{1\le k,l\le K+1}$, whose components are given as follows: 
\begin{align*}
    a_{kl}:=\l\{
        \begin{array}{cc}
        \dis 1-\frac{1}{\sqrt{2\a}}a^{f}_{\a,0}, &k=l,\\
        \dis -\frac{1}{\sqrt{2\a}}\left(a^{f}_{\a,k-l}-a^{f}_{\a,k-l-1}\r),&k>l,\\
        0,&k<l.
        \end{array}\r..
\end{align*}
Then the matrix $A^f_K$ is invertible, and it holds that 
\begin{align*}
\mb{a}^{G}_{\a,K}= (A^f_{K})^{-1} \mb{a}^{F}_{\a,K}.
\end{align*}
\end{lemma}

Using this system, we can construct a Laguerre expansion of $\ol{G}_q(x)$: 
for a vector $\mb{\vp}_{\a,K} = (\vp_{\a,0},\vp_{\a,1},\dots,\vp_{\a,K})^\top$, 
\begin{align}
\ol{G}_{q,K}(x):= \sum_{k=0}^K a_{\a,k}^G \vp_{\a,k}(x) = \mb{a}^{G}_{\a,K}\cdot \mb{\vp}_{\a,K} ,\quad x\in \R_+,  \label{Gq-bar}
\end{align}

\begin{lemma}\label{lem:tail}
If tail function of the L\'evy measure $\nu$, say $\ol{\nu}(x):=\int_x^\infty \nu(\df z)$, admits the derivatives up to order $r > 1$ and that {\red the distribution $G_q$ given in \eqref{Gq} has moments of any polynomial order.} Moreover,  there exists a constant $\k>0$ such that 
\begin{align}
\p_x^m \ol{\nu}(x) = O(1 + x^\k), \quad x\to \infty, \label{nu-bar}
\end{align}
for any $0\le m \le r-2$. 
Then we have $\ol{G}_q \in {\cal W}_\a^r(\R_+)$, and therefore 
\begin{align}
\sup_{x\in \R_+}\l|\ol{G}_{q,K}(x) - \ol{G}_q(x)\r|  = O\l(K^{-(r-1)/2}\r) \to 0,\quad K\to \infty. \label{Gq-uniform}
\end{align}
Furthermore, if the condition \eqref{nu-bar} is much milder such that
\begin{align}
\p_x^m \ol{\nu}(x)  = O\l( e^{\k x}\r), \quad x\to \infty, \label{nu-bar2}
\end{align}
Then, the consequence \eqref{Gq-uniform} also holds if we choose $\a > 2\k$. 
\end{lemma}

\begin{proof}
See Section \ref{proof:tail}.
\end{proof}

{\red 
\begin{remark}\label{rem:alpha}
Lemma \ref{lem:tail} tells us the meaning of the tuning parameter $\a>0$ in the Laguerre function. 
When we consider a model  where the L\'evy measure satisfies \eqref{nu-bar2}, we can take the value of $\a$ so that $a>2\k$. 
However, most standard cases where \eqref{nu-bar} holds true, we can choose any $\a>0$, in particular, $\a =1$ is simple and enough. 
\end{remark}
}

\subsection{Laguerre-type expansion for $q$-scale functions}
We can obtain a series expansion of the $q$-scale function by replacing $\ol{G}_q(z)$ in the expression of $W^{(q)}$ in Lemma \ref{lem:fs11} with the corresponding Laguerre expansion \eqref{Gq-bar}: for any $\a>0$ and $K\in \N$, 

\begin{defn}\label{def:WK}
For any $K\in \N$ and $\a>0$, the $K$th-Laguerre-type expansion of $W^{(q)}$, say $W^{(q)}_K$, 
is given by 
\begin{align}
W^{(q)}_K(x) := P(x;p,\g,D) - \mb{Q}_{\a,K}(x;p,\g,D)\cdot \mb{a}^G_{\a,K}, \label{W-approx}
\end{align}
where $\mb{Q}_{\a,K}:=\l(Q_{\a,0}, Q_{\a,1},\dots Q_{\a,K}\r)^\top$ with  
\begin{align*}
P(x;p,\g,D)&:=\l\{
    \begin{aligned} 
    &\frac{e^{\g x}-e^{-\b x}}{D(1-p)(\b+\g)},&D>0,\\
    &\frac{e^{\g x}}{c(1-p)},&D=0,
    \end{aligned}\r. \\
Q_{\a,k}(x;p,\g,D)&:=\l\{
   \begin{aligned}
   &\frac{\g \Psi_{\a,k}(x;\g)+\b\Psi_{\a,k}(x;-\b)}{D(1-p)(\b+\g)},&D>0,\\
   &\frac{\vp_{\a,k}(x)+\g \Psi_{\a,k}(x;\g)}{c(1-p)},&D=0
   \end{aligned}\r.
\end{align*} 
and 
\begin{align}
\Psi_{\a,k}(x;b):=\int_{0}^x e^{b(x-z)}\vp_{\a,k}(z)\,\df z,\quad b \in \R. \label{Psi}
\end{align}
\end{defn}

\begin{remark}
There is alternative version of the $q$-scale function $Z^{(q)}:\R \to [1,\infty)$, defined as 
\begin{align*}
Z^{(q)}(x) = 1 + q\int_0^x W^{(q)}(z)\,\df z, \quad x\in \R, 
\end{align*}
{\red where we regard that $\int_0^x = 0$ if $x<0$. Note that $W^{(q)}(x) = q^{-1} \p_x Z^{(q)}(x)$ as $q\ne 0$. }
The Laguerre type expansion of $Z^{(q)}$ is also defined as
\begin{align}
Z^{(q)}_K(x) &= 1 + q \int_0^x W^{(q)}_K(z)\,\df z \notag\\
&= 1 + q\l[ P^*(x; p,\g,D) - \mb{Q}^*_{\a,K}(x;p,\g,D)\cdot \mb{a}^G_{\a,K}\r], \label{Z-approx}
\end{align}
where $\mb{Q}^*_{\a,K}:=\l(Q^*_{\a,0}, Q^*_{\a,1},\dots Q^*_{\a,K}\r)^\top$ with  
\begin{align*} 
P^*(z;p,\g,D)&:=\l\{
    \begin{aligned}
    &\frac{\g^{-1}(1 - e^{\g x})-\b^{-1}(1-e^{-\b x})}{D(1-p)(\b+\g)},&D>0,\\
    &\frac{1 - e^{\g x}}{\g c(1-p)},&D=0,
    \end{aligned}\r. \\
Q^*_{\a,k}(x;p,\g,D)&:=\l\{
   \begin{aligned}
   &\frac{\Psi_{\a,k}(x;\g) - \Psi_{\a,k}(x;-\b)}{D(1-p)(\b+\g)},&D>0,\\
   &\frac{\Psi_{\a,k}(x;\g)}{c(1-p)},&D=0
   \end{aligned}\r.
\end{align*} 
According to Lemma \ref{lem:sz19}, if $\ol{G}_q \in {\cal W}_\a^r(\R_+)$ (e.g., $f:=\ol{G}_q$ satisfies \eqref{SL-cond}), then 
\[
\sup_{x\in \R_+}|W^{(q)}_K(x) -  W^{(q)}(x)| \to 0,\quad K\to \infty;  
\]
see \eqref{Gq-uniform}. Therefore, it follows for any compact sets $V\subset \R_+$ that 
\[
\sup_{x\in V}|Z^{(q)}_K(x) -  Z^{(q)}(x)| \to 0,\quad K\to \infty.  
\]
\end{remark}

\begin{remark}
The approximation formulas obtained in equations \eqref{W-approx} and \eqref{Z-approx} are essentially different from those in Xie \textit{et al.} \cite{xcz24}, and being more elementary. Their approximation focuses on the relationship between the probability density of the ``killed process" $X_{e_t}$ (where $e_t$ is an exponential random variable with mean 1) and its scale function, cleverly expanding the probability density into a Laguerre series. In our approximation, we utilize a Laguerre series expansion for the compound geometric distribution $\overline{G}_q$, essentially similar to the approach of Shimizu and Zhang \cite{sz19}. The utility of this approach becomes apparent in the subsequent part of statistical inference, where similar types of estimators as in \cite{sz19} are directly available and shown to be asymptotically normal, demonstrating an advantage of our method.
\end{remark}

{\red 
\section{Statistical inference: Main theorems}\label{sec:inference}
}

We will proceed with the statistical estimation of scale functions. 
When considering the practical use of $q$-scale functions, we usually {\red have} applications to actuarial science, where $X=(X_t)_{t\ge 0}$ is the dynamics of an insurance surplus model. 
Then, the coefficient of the linear term $c$ corresponds to the premium rate, so we specifically treat the value of $c$ as known in this paper.

\subsection{Sampling scheme}\label{sec:sampling}
Let $n\in \N$. 
We assume that the process $X=(X_t)_{t\ge 0}$ is observed at discrete time points $t_i^n:= i\D_n\ (i=0,1,\dots,n)$ for some number $\D_n>0$: 
\[
\mb{X}^n:=\{X_{t_i^n}\,|\, i = 0,1,\dots, n\}; \quad T_n:=n\D_n.  
\]
In particular, the initial value $X_{t_0^n} = x$ is assumed to be known. 
Moreover, we also assume that ``large" jumps of $X$ are observed, that is, for some number $\e_n>0$, 
we can identify the jumps whose sizes are larger than $\e_n$, which are finitely many on $[0,T_n]$: 
\[
\mb{J}^n:= \{\D X_t:= X_t - X_{t-} \,|\, t\in [0,T_n],\ |\D X_t| > \e_n \}. 
\]

Assume that we have data $\mb{X}^n \cup \mb{J}^n$, and consider the following asymptotics: 
\begin{align}
\D_n\to 0;\quad T_n\to \infty;\quad \e_n\to 0, \label{obs}
\end{align}
as $n\to \infty$. In the sequel, we always consider the limit $n\to \infty$ when considering the asymptotic symbols, 
and assume \eqref{obs} for the sampling scheme $(\D_n,T_n,\e_n)$ without specially mentioned. 

{\red 
\begin{remark}
The data $\mb{X}^n$ are assumed to be the data on the remaining amount of reserves that insurance companies record on a regular basis. On the other hand, $\mb{J}^n$ are assumed to be ``large" claims. 
It may seem unnatural to consider such a model with infinitely many jumps when considering insurance surplus, but this is a standard surplus approximation in risk theory. Arguing under asymptotics such as \eqref{obs} with data like $\mb{J}^n$ is a usual way of theoretically justifying that the more detailed claims data you collect as much as possible, the better the estimation. As a practical matter, it is not necessary to be able to observe infinitely many ``small" jumps in reality.
\end{remark}
}

We shall make some assumptions on the  scheme $(\D_n,T_n,\e_n)$: 
\bi
\item[{\bf S1}]: $n\D_n^2\to  0$. 
\item[{\bf S2}]: $\dis \int_0^{\e_n} z\,\nu(\df z) + \int_0^{\e_n} z^2\,\nu(\df z) = o(T_n^{-1/2})$.  
\ei
In addition, to ensure some integrability with respect to $\nu$, we will prepare the following moment conditions on $\nu$:  
\bi
\item[{\bf M{[$k$]}}]: For some given $k>0$, there exists some $\e\in (0,1)$ such that $\nu( |z|\vee |z|^{k + \e}) < \infty$. 
\ei 

\subsection{Main theorems}\label{sec:main}

For the statistical argument, we denote by $p_0, D_0$ and $\g_0$ the true values of parameters $p, D$ and $\g$, respectively. 
We will give estimators for approximations $W_K^{(q)}$ and $Z_K^{(q)}$, in which the parameters $p,D$ and $\g$ are replaced by their truth.  

According to the expressions \eqref{W-approx} and \eqref{Z-approx}, we construct estimators of $W_K^{(q)}$ and $Z_K^{(q)}$ as follows: 
\begin{align*}
\wh{W}^{(q)}_{K}(x)&:=P(x;\wh{p}_{n},\wh{\g}_{n},\wh{D}_{n}) -\mb{Q}_{\a,K}(x;\wh{p}_{n},\wh{\g}_{n},\wh{D}_{n})\cdot \wh{\mb{a}}^{G}_{\a,K}; \\
\wh{Z}^{(q)}_{K}(x)&:=q \l[P^*(x;\wh{p}_{n},\wh{\g}_{n},\wh{D}_{n}) -\mb{Q}^*_{\a,K}(x;\wh{p}_{n},\wh{\g}_{n},\wh{D}_{n})\cdot \wh{\mb{a}}^{G}_{\a,K}\r]. 
\end{align*}

The consistency and asymptotic normality for $\wh{W}^{(q)}_K$ are obtained for each $K\in \N$ as follows. 

\begin{thm}\label{thm:W}
Suppose the assumptions NPC, S1, S2, and M[2]. Then it holds true for any $q>0$, $K\in \N$ and $x \in \R_+$ that 
\[
\wh{W}^{(q)}_K(x) \toP W^{(q)}_K(x). 
\]
In particular, as $q=0$, we have the uniform consistency: 
\begin{align*}
\sup_{x\in\R_+}\l|\wh{W}^{(0)}_K(x)-W^{(0)}_K(x)\r|\toP 0. 
\end{align*}
In addition, suppose further M[4]. Then we have 
\[
\sqrt{T_{n}}\l(\wh{W}^{(q)}_{K}(x)-W^{(q)}_{K}(x)\r)\toD N\l(0,\sigma_{K}(x)\r), 
\]
where $\sigma_{K}(x):=[C_{K}(x)\G_{K}]\S_{K}[C_{K}(x)\G_{K}]^{\top}$ with 
$\S_{K}$ and $\G_{K}$ given in \eqref{S-K} and \eqref{G-K}),  respectively, and $C_K$ is the $(2K+4)$-dim vector given by 
\[
C_{K}(x):=\l(
\mb{Q}_{\a,K}(x;p_0,\g_0,D_0)^\top(A_K^f)^{-1}B_{K},
\p_{(p,\g)}\l[P(x;\g_0,p_0,D_0)-\mb{Q}_{\a,K}(x;p_0,\g_0,D_0)\cdot \mb{a}^{G}_{\a,K}\r]\r)^\top ,
\]
with the matrix $B_K$ given in Corollary \ref{cor:sz-G}. 
\end{thm}

\begin{proof}
See Section \ref{proof:W}. 
\end{proof}

\begin{thm}\label{thm:Z}
Suppose the assumptions NPC, S1, S2, and M[2]. Then it holds true for any $q>0$, $K\in \N$ and $x \in \R_+$ that 
\[
\wh{Z}^{(q)}_K(x) \toP Z^{(q)}_K(x). 
\]
In particular, as $q=0$, we have the uniform consistency: for any compact set $V\subset \R_+$, 
\begin{align*}
\sup_{x\in V}\l|\wh{Z}^{(0)}_K(x)-Z^{(0)}_K(x)\r|\toP 0. 
\end{align*}
In addition, suppose further M[4]. Then we have 
\[
\sqrt{T_{n}}\l(\wh{Z}^{(q)}_K(x)-Z^{(q)}_K(x)\r)\toD N\l(0,\sigma^*_{K}(x)\r), 
\]
where $\sigma^*_{K}(x):=q^2[C^*_K(x)\G_K]\S_K[C^*_K(x)\G_K]^\top$ with 
$\S_K$ and $\G_K$ given in \eqref{S-K} and \eqref{G-K}),  respectively, and $C^*_K$ is the $(2K+4)$-dim vector given by 
\[
C^*_K(x):=\l(
\mb{Q}^*_{\a,K}(x;p_0,\g_0,D_0)^\top(A_K^f)^{-1}B_{K},
\p_{(p,\g)}\l[P^*(x;\g_0,p_0,D_0)-\mb{Q}^*_{\a,K}(x;p_0,\g_0,D_0)\cdot \mb{a}^{G}_{\a,K}\r]\r)^\top ,
\]
with the matrix $B_K$ given in Corollary \ref{cor:sz-G}. 
\end{thm}

\begin{proof}
See Section \ref{proof:Z}. 
\end{proof}

\begin{cor}
\label{cor:WqK_ZqK_an}
Suppose the assumptions NPC, S1, S2, and M[2]. Then it holds true for any $q>0$, $K\in \N$ and $x \in \R_+$ that 
\begin{align*}
\sqrt{T_{n}}\l(
    \begin{matrix}
    \wh{W}^{(q)}_{K}(x)-W^{(q)}_{K}(x)\\
    \wh{Z}^{(q)}_{K}(x)-Z^{(q)}_{K}(x)
    \end{matrix}
\r)\toD N_{2}\l(\mathbf{0}_{2},\wt{\S}_{K}(x)\r),
\end{align*}
where 
\begin{align*}
\wt{\S}_{K}(x):=\l[\l(
    \begin{matrix}
    C_{K}(x)\\
    qC^{*}_{K}(x)
    \end{matrix}
\r)\G_{K}\r]\S_{K}\l[\l(
    \begin{matrix}
    C_{K}(x)\\
    qC^{*}_{K}(x)
    \end{matrix}
\r)\G_{K}\r]^{\top}.
\end{align*}
\end{cor}

\begin{proof}
See Section \ref{proof:WqK_ZqK_an}. 
\end{proof}

\begin{remark}
Our asymptotic results for $\wh{W}_K^{(q)}$ and $\wh{Z}_K^{(q)}$ are all for a fixed $K\in \N$, and one may concern with the case where $K=K_n\to \infty$ as well as $n\to \infty$. It is not a straightforward extension, even for consistency. For example, if we could show that, for each $x\in \R_+$, 
\[
\sup_{K\in \N}|\wh{W}_K^{(q)}(x) - W_K^{(q)}(x)| \toP 0,\quad n\to \infty, 
\]
then we can exchange the order of the limits $n\to \infty$ and $K\to \infty$, which concludes that 
 $\wh{W}_\infty^{(q)}(x) \toP W_\infty^{(q)}(x)$ for each $x\in \R_+$. On this point, we need further study. 
 As for the extension of the asymptotic normality, we need a more complicated discussion for extending to the case where $K$ depends on $n$ with $K=K_n\to \infty$. It will become a high-dimensional setting, and we need a high-dimensional central limit theorem for triangular arrays without even a martingale property. A sophisticated CLT would still need to be proven. 
\end{remark}

\noindent {\bf\large Acknowledgement.} This work is partially supported by JSPS KAKENHI Grant-in-Aid for Scientific Research (C) \#21K03358 Japan Science and Technology Agency CREST \#JPMJCR2115. Also, the authors sincerely thank the anonymous reviewers for their insightful comments, which have enhanced the quality of this paper. 

\ \vspace{7mm}\\
\noindent {\bf\Large Appendix}\appendix
\section{Some auxiliary statistics}

\subsection{Estimator of $D_0$}

For $D=\s^2/2$, we consider the following estimator proposed by Jacod \cite{j07} and Shimizu \cite{s11, s21}: for each fixed $T>0$, 
\[
{\red \wh{D}^{T}_{n}:=\frac{1}{2T}\l(\sum_{i=1}^{\lfloor T\D_{n}^{-1}\rfloor}|X_{i\D_{n}}-X_{(i-1)\D_{n}}|^{2}-\sum_{s\le T}|\D L_{s}|^{2}\I_{\{\D L_{s}>\e_n\}}\r). }
\]

\begin{lemma}[Shimizu \cite{s11}, Remark 3.2]\label{lem:D-est}
Under the assumptions S1 and S2, the estimator $\wh{D}^T_n$ is consistent to $D_0$ with the rate of convergence being faster than $\sqrt{T_n}$ such that, for any $t>0$, 
\begin{align*}
\sqrt{T_n}(\wh{D}^T_n - D_0) \toP 0. 
\end{align*} 
\end{lemma}

\begin{remark}
Since the constant $T>0$ in the estimator $\wh{D}^T_n$ can be arbitrary, we will fix it to be $T=1$ without loss of generality, and put 
\[
\wh{D}_n := \wh{D}_n^1, 
\]
in the sequel. In practice, we should appropriately choose the value of $T$ by looking at the number of data and the size of $\D_n$.
\end{remark}

\subsection{Estimator of $\nu$-functionals}

First, we would like to estimate the integral-type functional  $\nu (\mb{H}_\th)$, where $\mb{H}_\th:\R_+\to \R^d$, is a $\nu$-integrable function with an unknown parameter $\th\in \Th$: 
\[
\nu(\mb{H}_\th) := \l( \int_0^\infty H_\th^{(1)}(z)\,\nu(\df z), \dots, \int_0^\infty H_\th^{(d)}(z)\,\nu(\df z)\r)
\]
where $\Th$ is an open and bounded subset of $\R^l$ for some $l\in \N$. 
Note that the parameter $\theta$ can be a variety of parameters depending on the context. 
For instance, we will find that the parameter $p:=\int_0^\infty \wt{f}_q(z)\,\df z$, the coefficients of Laguerre expansion, e.g., $a_{\a,k}^f$ and $a_{\a,k}^F$, are all written in terms of the integral functional of $\nu$, later.

In short, we need to estimate the parameters $(D_0,\g_0, \nu(\mb{H}_{\th_0}))$, where $\th_0$ is the true parameter of $\th$. 
Hereafter, we assume that there exists an open and bounded set $\Th_1$ and $\Th_2$ of {\red $\R_+$} such that 
\[
(D_0,\g_0,\th_0)\in \Th_1\times \Th_2. 
\]
Moreover, we make the following assumptions on an integrands $\mb{H}_\th$, which are applied to a variety of $\mb{H}_\th$, locally in this section.
\bi
\item[H1{[$\d$]}]: {\red For each $\th\in \Theta$, there exists a $\d\ge 0$} such that $\nu(|\mb{H}_\th| \vee|\mb{H}_\th|^{2+\d}) < \infty$ 

\item [H2]: $\dis \sup_{\th \in \ol{\Th}} \nu(|\mb{H}_\th| \vee |\mb{H}_\th|^2) < \infty$. 

\item[H3]: There exists a $\nu$-integrable function $h_1:\R_+\to \R$ such that 
\[
\sup_{\th \in \ol{\Th}} |\mb{H}_\th(z)| \le h_1(z)
\]

\item[H4]: There exists a function $h_2:\R_+\to \R$ with $\nu(h_2\vee h_2^2) < \infty$ such that, for any $\k \in \R^l$, 
\[
\sup_{\th \in \ol{\Th}} |\mb{H}_{\th + \k}(z) - \mb{H}_\th(z)| \le h_2(z)|\k|. 
\] 
\item[H5]: For each $i=1,\dots,d$, 
\[
\int_0^{\e_n} H_\th^{(i)}(z)\,\nu(\df z) = o(T_n^{-1/2}). 
\]
\ei

As for functionals $\nu(\mb{H}_\th)$, we can use the following threshold-type estimator: 
\[
\wh{\nu}_n(\mb{H}_\th) := \frac{1}{T_n} \sum_{t \in (0, T_n]} \mb{H}_\th(\D L_t) \I_{\{\D L_t > \e_n\}}
\]

\begin{lemma}[Shimizu \cite{s11}, Propositions 3.1 and 3.2] \label{lem:nu-est}
\bi
\item[(i)] Under the assumption H1[0], we have 
\[
\wh{\nu}_n(\mb{H}_\th) \toP \nu(\mb{H}_\th),\quad \th \in \Th. 
\]
In addition, assuming further H2 and H4, we have the uniform consistency: 
\[
\sup_{\th \in \ol{\Th}} |\wh{\nu}_n(\mb{H}_\th) - \nu(\mb{H}_\th)| \to 0. 
\]

\item[(ii)] Under H1[$\d$] for some $\d>0$ and H5, we have 
\[
\sqrt{T_n}(\wh{\nu}_n(\mb{H}_\th) - \nu(\mb{H}_\th)) \toD N_d(\mb{0}_d,\S_\th), \quad \th \in \Th, 
\]
where $\S_\th = \l(\nu(H_\th^{(i)}H_\th^{(j)})\r)_{1\le i,j\le d}$. 

\ei
\end{lemma}

It will be easy to see the following version of the continuous mapping-type theorem holds for the estimator $\wh{\nu}_n(\mb{H}_\th)$: 
\begin{cor}\label{cor:nu-est}
Under the assumptions H1[0], H2--H4, it follows that 
\[
\wh{\nu}_n(\mb{H}_{\wh{\th}_n}) \toP \nu(\mb{H}_{\th_0}), 
\]
for any random sequence such that $\wh{\th}_n \toP \th_0\in \Th$. 
\end{cor}

\begin{proof}
Note that there exists a sub-sub sequence $\{\wh{\th}_{n'}\}$ for any subsequence of $\{\wh{\th}_n\}$ such that $\wh{\th}_{n'} \to \th_0\ a.s.$. 
Then, under H3, we can apply the Lebesgue convergence theorem to obtain that $\nu(\mb{H}_{\wh{\th}_{n'}})\to\nu(\mb{H}_{\th_0})\ \ a.s.$
That is, the sequence $\{\nu(\mb{H}_{\wh{\th}_n})\}$ has a sub-subsequence that converges to $\nu(\mb{H}_{\th_0})$ almost surely, 
which implies that $\nu(\mb{H}_{\wh{\th}_n})\toP\nu(\mb{H}_{\th_0})$. 
Moreover, since $\P(\wh{\th}_n\notin\Th)\to 0$, it follows from Lemma \ref{lem:nu-est} that, for any $\e>0$, 
\begin{align*}
\P\l(|\wh{\nu}_{n}(\mb{H}_{\wh{\th}_n})-\nu(\mb{H}_{\th_0})|>\e\r)
&\le \P\l(\sup_{\th\in \ol{\Th}} |\wh{\nu}_{n}(\mb{H}_\th)-\nu(\mb{H}_\th)|>\e/2,\ \wh{\th}_n\in\Th\r)\\
&\qquad +\P\l(|\nu(\mb{H}_{\wh{\th}_n})-\nu(\mb{H}_{\th_0})|>\e/2,\ \wh{\th}_n\in\Th\r) \\
&\qquad + \P\l(\wh{\th}_n\notin\Th\r) \to 0. 
\end{align*}
This completes the proof. 
\end{proof}

\subsection{Estimator of $\g_0$} 

An estimator of the Lundberg exponent $\g_0 = \Phi(q)$ is found in Shimizu \cite{s11} as an $M$-estimator as follows: 
\begin{align*}
\wh{\g}_n = \I_{\{q > 0\}}\cdot \arg\inf_{r \in \Th_2}\l|c r + \wh{D}_n r^2 - \wh{\nu}_n(k_r) - q\r|^2, 
\end{align*}
where $k_r(z) = e^{-r z} - 1$. {\red This estimator is quite natural because the contrast function is a direct estimator of the Lundberg equation ``$\psi_X(r) - q =0$", and it is useful because it satisfies consistency and asymptotic normality as follows: }

\begin{lemma}[Shimizu \cite{s11}, Lemma 3.3]\label{lem:g-est}
Suppose the conditions NPC, S1, S2, and M[2] hold true. 
Then we have 
\[
\sqrt{T_n}(\wh{\g}_n - \g_0) \toD N_1(0,v_0^2), 
\]
where 
\[
v_0^2 = \frac{\nu(k_{\g_0}^2)}{ \l(c + 2D_0\g_0  + \nu(\p_r k_{\g_0})\r)^2}. 
\]
In particular, we have the representation (in the proof of Lemma 3.3 in \cite{s11}) that 
\begin{align}
\sqrt{T_n}(\wh{\g}_n - \g_0) = \sqrt{T_n}(\wh{\nu}_n(\wt{H}_{\g_0}) - \nu(\wt{H}_{\g_0})) + o_p(1). \label{H-g}
\end{align}
where $\wt{H}_r(z):= k_r(z) / \p_z \psi_X(z)$ and $k_r(z) = e^{-rz} - 1$.  

\end{lemma}

\section{Estimators of $p_0$, $\mb{a}^f_{\a,K}$, $\mb{a}^F_{\a,K}$ and $\mb{a}^{G}_{\a,K}$. }

First, we prepare some notation to give representations of $p_0$, $\mb{a}^f_{\a,K}$ and $\mb{a}^F_{\a,K}$ in terms of $\nu$-functionals. 

Hereafter we put  the parameter $\th := (D,\g) \in \ol{\Th} := \ol{\Th}_1\times \ol{\Th}_2$, and assume that the true values 
\[
\th_0 := (D_0,\g_0) \in \Th := \Th_1\times \Th_2, 
\]
where $\Th_1$ and $\Th_2$ are open and bounded subsets of $\R_+$. 

We define the following notation: As $D> 0$, for $\a>0$ and $k\in \N_0$, 
\begin{align*}
H_p(z;\th) &:= D^{-1} \int_0^z \df y \int_y^\infty e^{-\b (x-y)}e^{-\g (z-x)}\,\df x, \\
H_{\a,k}^f(z;\th)&:= D^{-1}\int_0^z \,\df y \int_y^\infty e^{-\b (x-y)}e^{-\g (z-x)}\vp_{\a,k}(x)\,\df x, \\
H_{\a,k}^F(z;\th)&:= D^{-1}\int_0^z \,\df y \int_y^\infty e^{-\b (x-y)}e^{-\g (z-x)}\Psi_{\a,k}(x;0)\,\df x. 
\end{align*}
Recall that $\Psi_{\a,k}(x;b)$ is given by \eqref{Psi} in Definition \ref{def:WK}. In particular, as $D=0$, 
\begin{align*}
H_p(z;\th) &:= c^{-1}\int_0^z e^{-\g (z-x)}\,\df x, \\
H_{\a,k}^f(z;\th)&:= c^{-1}\int_0^z e^{-\g (z-x)}\vp_{\a,k}(x)\,\df x, \\
H_{\a,k}^F(z;\th)&:= c^{-1}\int_0^z e^{-\g (z-x)}\Psi_{\a,k}(x;0)\,\df x. 
\end{align*}
As a convention in this paper, e.g., $\mb{H}_{\a,K}^{f} = (H_{\a,0}^f, \dots, H_{\a,K}^f)^\top$, among others. 

Then, it is straightforward to obtain the following expression by direct computations. 

\begin{lemma}\label{lem:paf}
It holds that 
\begin{align*}
p_0 = \nu(H_p(\cdot;\th_0)),\quad a_{\a,k}^{f}= \nu(H_{\a,k}^f(\cdot;\th_0)),\quad a_{\a,k}^F = \nu(H_{\a,k}^F(\cdot;\th_0)). 
\end{align*}
\end{lemma}

\begin{proof}
Use Fubini's theorem. We shall compute only $H_{\a,k}^F(z;\th)$: 
\begin{align*}
a^{F}_{\alpha,k}&:=\langle p\overline{F}_{q},\varphi_{\alpha,k}\rangle\\
&=\int_{0}^{\infty}p\overline{F}_{q}(x)\varphi_{\alpha,k}(x)\,\df x\\
&=\int_{0}^{\infty}p f_{q}(x)\left(\int_{0}^{x}\varphi_{\alpha,k}(u)\mathrm{d}u\r)\,\df x \qquad \mbox{(by integration-by-parts)}\\
&=\int_{0}^{\infty}\left(D^{-1}\int_{0}^{x}\mathrm{d}y\int_{y}^{\infty}e^{-\beta(x-y)}e^{-\g(z-y)}\nu(\mathrm{d}z)\r)\Psi_{\alpha,k}(x;0)\,\df x\\
&=\int_{0}^{\infty}\left(D^{-1}\int_{0}^{z}\mathrm{d}y\int_{y}^{\infty}e^{-\beta(x-y)}e^{-\g(z-y)}\Psi_{\alpha,k}(x;0)\,\df x\r)\nu(\mathrm{d}z)\\
&=\nu\l(H^{F}_{\a,k}(\cdot;\th)\r).
\end{align*}
The others are similar and omitted. 
\end{proof}

Thanks to this lemma, we have estimators of $p_0$, $a^f_{\a,k}$ and $a^F_{\a,k}$ as follows: 
\begin{align}
\wh{p}_n = \wh{\nu}_n(H_p(\cdot;\wh{\th}_n)),\quad \wh{a}_{\a,k}^f = \wh{\nu}_n(H_{\a,k}^f(\cdot;\wh{\th}_n)),\quad \wh{a}_{\a,k}^F = \wh{\nu}_n(H_{\a,k}^F(\cdot;\wh{\th}_n)), 
\end{align}
where 
\[
\wh{\th}_n = (\wh{D}_n, \wh{\g}_n)^\top. 
\]

\begin{thm}\label{thm:consist}
Suppose the assumptions NPC, S1, S2, and M[2]. Then the estimators $\wh{p}_n$, $ \wh{a}_{\a,k}^f$ and $\wh{a}_{\a,k}^F$ are consistent to their true values. In particular, it follows for each $\a>0$ and $K\in \N$ that 
\[
\l(\wh{\mb{a}}_{\a,K}^f, \wh{\mb{a}}_{\a,K}^F,\wh{p}_n\r)^\top \toP  \l(\mb{a}_{\a,K}^f, \mb{a}_{\a,K}^F,p_0\r)^\top. 
\]
where $p_0$ is the true value of $p=\int_0^\infty \wt{f}_q(z)\,\df z$. 
\end{thm}

\begin{proof}
According to Lemma \ref{lem:paf}, $\wh{p}_n, \wh{a}_{\a,k}^f$ and $\wh{a}_{\a,k}^F$ are all represented by $\nu$-functionals, for which we can apply Corollary \ref{cor:nu-est} since $\wh{\th}_n = (\wh{D}_n,\wh{\g}_n) \toP \th_0=(D_0,\g_0)$ under our assumptions by Lemmas \ref{lem:D-est} and \ref{lem:g-est}. Therefore, we can check H1[0], H2--H4-type conditions for each $H_\th:=H_p$, $H_{\a,k}$ and $H_{\a,k}^F$, 
but it is straightforward under M[2] since H2 and H3 are true by Lemma \ref{lem:pH}, and H4 is also true by Lemma \ref{lem:gH}. 
\end{proof}

To state the asymptotic normality result, we define the following notation: 
\begin{align*}
\mb{H}_{\a,K}(z;\th)&:=(\mb{H}^{f}_{\a,K}(z;\th),\mb{H}^{F}_{\a,K}(z;\th),H_{p}(z;\th))^{\top} \in \R^{2K+3}; \\
\wt{\mb{H}}_{\a,K}(z;\th)&:=(\mb{H}_{\a,K}(z;\th),\wt{H}_{\g}(z))^{\top} \in \R^{2K+4},  
\end{align*}
where $\wt{H}_\g$ is given in Lemma \ref{lem:g-est}, \eqref{H-g}. 

\begin{thm}\label{thm:asymNormal}
Suppose the same assumptions as in Theorem \ref{thm:consist} and M[4]. 
Then we have the asymptotic normality : for each $\a>0$ and $K\in \N$, 
\begin{align*}
\sqrt{T_n}
    \begin{pmatrix}
    \wh{\mb{a}}^{f}_{\a,K}-\mb{a}^{f}_{\a,K}\\
    \wh{\mb{a}}^{F}_{\a,K}-\mb{a}^{F}_{\a,K}\\
    \wh{p}_{n}-p_0\\
    \wh{\g}_{n}-\g_0
    \end{pmatrix}
\toD N_{2K+4}\l(\mathbf{0}_{2K+4},\G_{K}\S_{K}\G_{K}^{\top}\r),
\end{align*}
where $\S_{K}:=(\s_{ij})_{1\leq i,j\leq 2K+4}$ with 
\begin{align}\label{S-K}
\s_{ij}:=\int_0^\infty \wt{H}_{\a,K}^{(i)}(z;\th_0)\wt{H}_{\a,K}^{(j)}(z;\th_0)\nu(\df z), 
\end{align}
and $\G_{K}$ is the $(2K+4)\times(2K+4)$-matrix denoted by 
\begin{align}\label{G-K}
\G_{K}:=
    \begin{pmatrix}
    I_{2K+3} & \nu(\p_{\g}\mb{H}_{\a,K}(\cdot;\th_0))\\
    \mathbf{0}_{2K+3}^\top  & 1
    \end{pmatrix}
\end{align}

\end{thm}

\begin{proof}
We only show the case where $D>0$ because the proof for $D=0$ is similar. 

Firstly, thanks to Taylor's formula, it follows that 
\begin{align*}
&\sqrt{T_n}\Big(\wh{\nu}_{n}(\mb{H}_{\a,K}(\cdot;\wh{\th}_{n})) - \nu(\mb{H}_{\a,K}(\cdot;\th_0))\Big)\\
&=\sqrt{T_n}\l(\wh{\nu}_{n}(\mb{H}_{\a,K}(\cdot;\wh{\th}_{n}))-\wh{\nu}_{n}(\mb{H}_{\a,K}(\cdot;\th_0))\r) +\sqrt{T_n}\l(\wh{\nu}_{n}(\mb{H}_{\a,K}(\cdot;\th_0))-\nu(\mb{H}_{\a,K}(\cdot;\th_0))\r)\\
&=\wh{\nu}_{n}\l(\int_0^{1}\p_{\g}\mb{H}_{\a,K}(\cdot;\th^{n}_{u})\,\df u\r)\sqrt{T_n}\l(\wh{\g}_{n}-\g_0\r) 
 +\wh{\nu}_{n}\l(\int_0^{1}\p_{D}\mb{H}_{\a,K}(\cdot;\th^{n}_{u})\,\df u\r)\sqrt{T_n}(\wh{D}_{n}-D_0)\\
&\qquad +\sqrt{T_n}\l(\wh{\nu}_{n}(\mb{H}_{\a,K}(\cdot;\th_0))-\nu(\mb{H}_{\a,K}(\cdot;\th_0))\r)\\
&=\begin{pmatrix}
    I_{2K+3} & \wh{\nu}_{n}\l(\int_0^{1}\p_{\g}\mb{H}_{\a,K}(\cdot;\th^{n}_{u})\,\df u\r)
    \end{pmatrix} 
    \sqrt{T_n}
    \begin{pmatrix}
    \wh{\nu}_{n}(\mb{H}_{\a,K}(\cdot;\th_0))-\nu(\mb{H}_{\a,K}(\cdot;\th_0))\\
    \wh{\g}_{n}-\g_0
    \end{pmatrix}\\
&\qquad +\l[ \l(\wh{\nu}_{n}\l(\int_0^{1}\p_{D}\mb{H}_{\a,K}(\cdot;\th^{n}_{u})\,\df u\r)\I_{\{\wh{\th}_{n}\in\Th\}}\r)
+\wh{\nu}_{n}\l(\int_0^{1}\p_{D}\mb{H}_{\a,K}(\cdot;\th^{n}_{u})\,\df u\r)\I_{\{\wh{\th}_{n}\notin\Th\}} \r] \times \\
&\qquad \times \sqrt{T_n}(\wh{D}_{n}-D_0) =:S_{1}+S_{2}, 
\end{align*}
where $\th^{n}_{u}:=\th_0+u(\wh{\th}_{n}-\th_0),\ (u\in(0,1))$. 

As for $S_2$, it follows from Lemma \ref{lem:pH} that, for any $k\in \N_0$, 
\begin{align*}
\l|\int_0^{1}\p_{D}H_{\a,k}(z;\th^{n}_{u})\,\df u\I_{\{\wh{\th}_{n}\in\Th\}}\r|
\leq\sup_{\th\in\Th}|\p_{D}H_{\a,k}(z;\th)|\lesssim z
\end{align*}
Therefore the condition M[4] ensures that 
\begin{align*}
\nu\l(\l|\int_0^{1}\p_{D}H_{\a,k}(\cdot;\th^{n}_{u})\,\df u\I_{\{\wh{\th}_{n}\in\Th\}}\r|^{2}\r)<\infty.
\end{align*}
Moreover, it follows for any $\e>0$ that 
\begin{align*}
\l\{\wh{\nu}_{n}\l(\int_0^{1}\p_{D}\mb{H}_{\a,K}(\cdot;\th^{n}_{u})\,\df u\r)\I_{\{\wh{\th}_{n}\notin\Th\}}>\e\r\}
\subset\{\I_{\{\wh{\th}_{n}\notin\Th\}}=1\}=\{\wh{\th}_{n}\notin\Th\}. 
\end{align*}
Since $\wh{\th}_{n}\toP \th_0\in\Th$ by S1, S2 and M[2], it follows from Lemmas \ref{lem:D-est} and \ref{lem:nu-est} that 
\begin{align*}
S_2=(O_p(1)+o_p(1))\sqrt{T_n}\l(\wh{D}_n-D_0\r) \toP 0. 
\end{align*}

Secondly, we show the following convergence: 
\begin{align}
\wh{\nu}_{n}\l(\int_0^{1}\p_{\g}\mb{H}_{\a,K}(\cdot;\th^{n}_{u})\,\df u\r)
\toP \nu(\p_{\g}\mb{H}_{\a,K}(\cdot;\th_0)). \label{nu-int-pH}
\end{align}
Thanks to Fubini's theorem, we have that 
\begin{align*}
\begin{split}
\bigg|\wh{\nu}_{n}\bigg(\int_0^{1}&\p_{\g}\mb{H}_{\a,K}(\cdot;\th^{n}_{u})\,\df u\bigg)-\nu(\p_{\g}\mb{H}_{\a,K}(\cdot;\th_0))\bigg|\\
&\leq\l|\wh{\nu}_{n}\l(\int_0^{1}\p_{\g}\mb{H}_{\a,K}(\cdot;\th^{n}_{u})\,\df u\r)
-\nu\l(\int_0^{1}\p_{\g}\mb{H}_{\a,K}(\cdot;\th^{n}_{u})\,\df u\r)\r|\I_{\{\wh{\th}_{n}\in\Th\}}\\
&\qquad +\l|\wh{\nu}_{n}\l(\int_0^{1}\p_{\g}\mb{H}_{\a,K}(\cdot;\th^{n}_{u})\,\df u\r)
-\nu\l(\int_0^{1}\p_{\g}\mb{H}_{\a,K}(\cdot;\th^{n}_{u})\,\df u\r)\r|\I_{\{\wh{\th}_{n}\notin\Th\}}\\
&\qquad +\l|\nu\l(\int_0^{1}\p_{\g}\mb{H}_{\a,K}(\cdot;\th^{n}_{u})\,\df u\r)-\nu(\p_{\g}\mb{H}_{\a,K}(\cdot;\th_0))\r|\\
&\leq\int_0^{1}\l|\wh{\nu}_{n}(\p_{\g}\mb{H}_{\a,K}(\cdot;\th^{n}_{u}))-\nu(\p_{\g}\mb{H}_{\a,K}(\cdot;\th^{n}_{u}))\r|\,\df u\I_{\{\wh{\th}_{n}\in\Th\}}\\
&\qquad +\l|\wh{\nu}_{n}\l(\int_0^{1}\p_{\g}\mb{H}_{\a,K}(\cdot;\th^{n}_{u})\,\df u\r)
-\nu\l(\int_0^{1}\p_{\g}\mb{H}_{\a,K}(\cdot;\th^{n}_{u})\,\df u\r)\r|\I_{\{\wh{\th}_{n}\notin\Th\}}\\
&\qquad +\int_0^{1}\l|\nu(\p_{\g}\mb{H}_{\a,K}(\cdot;\th^{n}_{u}))-\nu(\p_{\g}\mb{H}_{\a,K}(\cdot;\th_0))\r|\,\df u\\
&=:S'_{1}+S'_{2}+S'_{3}
\end{split}
\end{align*}

As for $S'_1$, since 
\begin{align*}
S'_1\leq\sup_{\th\in\Th}\l|\wh{\nu}_{n}(\p_{\g}\mb{H}_{\a,K}(\cdot;\th))-\nu(\p_{\g}\mb{H}_{\a,K}(\cdot;\th))\r|
\end{align*}
we see from Lemmas \ref{lem:nu-est} and \ref{lem:gH} that $S'_1\toP 0$. 

As for $S'_2$, since  
\begin{align*}
\l\{|S'_{2}|>\e'\r\} \subset\{\I_{\{\wh{\th}_{n}\notin\Th\}}=1\}=\{\wh{\th}_{n}\notin\Th\}, 
\end{align*}
we have $S'_2\toP 0$ by the consistency of $\wh{\th}_n$.  

As for $S'_3$, note that 
\begin{align*}
S'_{3}&\leq\int_0^{1}\l|\nu(\p_{\g}\mb{H}_{\a,k}(\cdot;\th^{n}_{u}))-\nu(\p_{\g}\mb{H}_{\a,k}(\cdot;\th_0))\r|\I_{\{|\th^{n}_{u}-\th_0|<\delta'\}}\,\df u\\
&\quad +\int_0^{1}\l|\nu(\p_{\g}\mb{H}_{\a,k}(\cdot;\th^{n}_{u}))-\nu(\p_{\g}\mb{H}_{\a,k}(\cdot;\th_0))\r|\I_{\{|\th^{n}_{u}-\th_0|\geq\delta'\}}\,\df u. 
\end{align*}
The first term on the right-hand side goes to zero in probability due to the continuity of the mapping $\th\mapsto\nu(\p_{\g}\mb{H}_{\a,K}(\cdot;\th))$, which is a consequence by the Lebesgue convergence theorem with Lemma \ref{lem:pH} and the condition M[2]. 
The second term on the right-hand side also goes to zero in probability by the consistency of $\th_u^n$ to $\th_0$. 
Hence we see that $S'_{3}\toP 0$, which concludes the convergence \eqref{nu-int-pH}. 

Finally, we can show that $\sqrt{T_n}\l(\wh{\nu}_{n}(\wt{\mb{H}}_{\a,K}(\cdot;\th_0))-\nu(\wt{\mb{H}}_{\a,K}(\cdot;\th_0))\r)$ is asymptotically normal 
by checking the conditions in Lemma \ref{lem:nu-est} with $\mb{H}_\th := \wt{\mb{H}}_{\a,K}(\cdot;\th)$. 
In fact, it follows from Lemma \ref{lem:pH} and M[4] that $\nu(|\mb{H}_\th|^4) < \infty$ for each $\th\in \Th$, and that 
H5 also holds true since, by Lemma \ref{lem:pH} and (S2), 
\begin{align*}
\int_0^{\e_{n}}
\l(H^{f}_{\a,k}(z;\th)+H^{F}_{\a,k}(z;\th)+H_{p}(z;\th)+H_{\g_0}(z)\r)\nu(\,\df z)
\lesssim \int_0^{\e_{n}}z\nu(\,\df z)
=o(T_n^{-1/2}),\quad n\to\infty.
\end{align*}
As a consequence, we see from Lemma \ref{lem:nu-est} that 
\begin{align}
\sqrt{T_n}\l(\wh{\nu}_{n}(\wt{\mb{H}}_{\a,K}(\cdot;\th_0))-\nu(\wt{\mb{H}}_{\a,K}(\cdot;\th_0))\r)\toD N_{2K+4}(\mathbf{0}_{2K+4},\S_{K}),
\label{eq:all_an}
\end{align}
where $\S_K$ is given in \eqref{G-K}. Now, putting 
\begin{align*}
\wh{\G}_{K}:=\l(
    \begin{matrix}
    I_{2K+3} & \wh{\nu}_{n}\l(\int_0^{1}\p_{\g}\mb{H}_{\a,K}(\ \cdot\ :\th_0)\,\df u\r)\\
    \mathbf{0}_{2K+3} & 1
    \end{matrix}
\r), 
\end{align*}
we have the expression 
\begin{align*}
\sqrt{T_n}\l(
    \begin{matrix}
    \wh{\mb{a}}^{f}_{\a,K}-\mb{a}^{f}_{\a,K}\\
    \wh{\mb{a}}^{F}_{\a,K}-\mb{a}^{F}_{\a,K}\\
    \wh{p}_{n}-p_0\\
    \wh{\g}_{n}-\g_0
    \end{matrix}
\r)&=\sqrt{T_n}\l(
    \begin{matrix}
    \wh{\nu}_{n}(\mb{H}_{\a,K}(\cdot;\wh{\th}_{n}))-\nu(\mb{H}_{\a,K}(\cdot;\th_0))\\
    \wh{\g}_{n}-\g_0 
    \end{matrix}
\r)\\
&=\wh{\G}_{K}\cdot
\sqrt{T_n}\l(\wh{\nu}_{n}(\wt{\mb{H}}_{\a,K}(\cdot;\th_0))-\nu(\wt{\mb{H}}_{\a,K}(\cdot;\th_0))\r)+o_{p}(1)
\end{align*}
Since $\wh{\G}_K\toP \G_K$ by \eqref{nu-int-pH}, Slutsky's lemma yields the consequence, and the proof is completed. 
\end{proof}

According to Shimizu and Zhang \cite{sz19}, Theorems 1 and 2 and their proofs, we find the following results. 
\begin{cor}\label{cor:sz-G}
Under the same assumptions as in Theorem \ref{thm:consist}, it holds that 
\[
\wh{\mb{a}}^{G}_{\a,K}\toP \mb{a}^{G}_{\a,K}. 
\]
In addition, suppose further the condition M[2], then $\wh{\mb{a}}^{G}_{\a,K}$ is asymptotically normal since 
\begin{equation}
\label{eq:aG_an}
\sqrt{T_{n}}\l(\wh{\mb{a}}^{G}_{\a,K}-\mb{a}^{G}_{\a,K}\r)=-(A_K^f)^{-1}\wh{B}_{K}\sqrt{T_{n}}\l(
    \begin{matrix}
    \wh{\mb{a}}^{f}_{\a,K}-\mb{a}^{f}_{\a,K}\\
    \wh{\mb{a}}^{F}_{\a,K}-\mb{a}^{F}_{\a,K}
    \end{matrix}
\r)
\end{equation}
where $\wh{B}_K$ is a consistent estimator of the the $(K+1)\times(2K+2)$ matrix $B_K$ given as follows:  
$B_K := (B^{*}_{K}, -I_{K+1})$, where $B^{*}_{K}:=(b_{kl})_{1\leq k,l\leq K+1}$, whose elements are given by 
\begin{align*}
b_{kl}:=\l\{
    \begin{array}{cc}
    \dis -\frac{1}{\sqrt{2\a}}a^{G}_{\a,0}, & k=l,\\
    \dis -\frac{1}{\sqrt{2\a}}\l(a^{G}_{\a,k-l}-a^{G}_{\a,k-l-1}\r), & k>l,\\
    \dis 0, & k<l.
    \end{array}\r.
\end{align*}
For example, an estimator $\wh{B}_K$ is obtained by replacing $a^{G}_{\a,k}$, each element of $B_K$, by its consistent estimator $\wh{a}^{G}_{\a,k}$. 
\end{cor}

\section{Auxiliary lemmas and some proofs}\label{sec:mis}

\subsection{On functions $f_{q}$, $F_{q}$ and $G_{q}$}\label{sec:fFG}

\begin{lemma} \label{lem:fq}
Under NPC, the probability density $f_q$ is well-defined and uniformly bounded.
\end{lemma}

\begin{proof}
We prove only the case where $D>0$, and the one for $D=0$ is similarly proved. 

First, we see from Lemma 3.2 in Feng and Shimizu \cite{fs13} that 
\[
p:= \int_0^\infty \wt{f}_q(z)\,\df z \in (0,1), 
\]
under the net profit condition NPC. Therefore $f_q:=p^{-1}\wt{f}_q$ is well-defined. 

Moreover, since  $e^{-\g(z-y)}<1$ for $y<z<\infty$ and $e^{-\beta(x-y)}\le 1$ for $0<x\le y$, 
it follows that $f_q$ is uniformly bounded as follows: for any $x\in \R$, 
\begin{align*}
\wt{f}_{q}(x)&=D^{-1}\int_{0}^{x}\df y\int_{y}^{\infty}e^{-\b(x-y)}e^{-\g(z-y)}\nu(\df z) \\
&\le D^{-1}\int_{0}^{x}\df y\int_{y}^{\infty}\nu(\df z) \\
&\le D^{-1}\int_{0}^{\infty}z\nu(\df z) < \infty. 
\end{align*}
\end{proof}

\begin{lemma}\label{lem:FqGq_int}
For $F_{q}$ and $G_{q}$, their tail functions satisfy that  $\ol{F}_{q},\ \ol{G}_{q}\in L^s(\R_{+})$ for any $s\ge 1$.
\end{lemma}

\begin{proof}
Since $0\le \ol{F}_{q}(x), \ol{G}_{q}(x)\le 1$ for any $x\in \R$, 
it suffices to show that $\ol{F}_{q},\ \ol{G}_{q}\in L^{1}(\R_{+})$. 

The Fubini theorem yields that 
\begin{align*}
\int_{0}^{\infty}\ol{F}_{q}(x)\,\df x&=\int_{0}^{\infty}\l(\int_{u}^{\infty}f_{q}(x)\df x\r)\,\df u =\int_{0}^{\infty}xf_{q}(x)\,\df x\\
&=\int_{0}^{\infty}\l(\int_{0}^{z}\,\df y\int_{y}^{\infty}xe^{-\b(x-y)}e^{-\g(z-y)}\,\df x\r)\,\nu(\df z)\\
&=\frac{1}{\b\g}\l(\frac{1}{\b}-\frac{1}{\g}\r)\int_{0}^{\infty}(1-e^{-\g z})\,\nu(\df z)
+\frac{1}{\b\g}\int_{0}^{\infty}z\,\nu(\df z)\\
&\le\b^{-2}\int_{0}^{\infty}z\,\nu(\df z)<\infty.
\end{align*}
Moreover, taking the Laplace transform of the DRE \eqref{DRE-Gq} for $G_q$, we have that 
\begin{align*}
\int_{0}^{\infty}e^{-sx}\ol{G}_{q}(x)\,\df x
=\frac{p\int_{0}^{\infty}e^{-sx}f_{q}(x)\,\df x}{1-p\int_{0}^{\infty}e^{-sx}\ol{F}_{q}(x)\,\df x}. 
\end{align*}
Substituting with $s=0$, we have $\ol{G}_{q}\in L^{1}(\R_{+})$ since $f_{q},\ol{F}_{q}\in L^{1}(\R_{+})$. 
\end{proof}

\subsection{On functions $H_p$, $H_{\a,k}^f$ and $H_{\a,k}^F$}

\begin{lemma}\label{lem:pH}
For each $z\in \R_{+}$ and $m=0,\ 1$,  the following inequalities hold true:  
\begin{align*}
\sup_{\th\in \ol{\Th}}|\p_{D}^{m}H_{p}(z;\th) | \lesssim z,\quad
\sup_{\th\in \ol{\Th}}|\p_{D}^{m}H^{f}_{\a,k}(z;\th)| \lesssim z,\quad
\sup_{\th\in \ol{\Th}}|\p_{D}^{m}H^{F}_{\a,k}(z;\th)| \lesssim z.
\end{align*} 

Moreover, it follows that 
\begin{align*}
\sup_{\th\in \ol{\Th}}|\p_{\g}H_{p}(z;\th)| \lesssim z,\quad
\sup_{\th\in \ol{\Th}}|\p_{\g}H^{f}_{\a,k}(z;\th)| \lesssim z,\quad
\sup_{\th\in \ol{\Th}}|\p_{\g}H^{F}_{\a,k}(z;\th)| \lesssim z + z^2
\end{align*}
\end{lemma}

\begin{proof}
We shall show the case where $D>0$. The case where $D=0$ is similarly proved. 

Since $\Th=\Th_1\times \Th_2$, each $\Th_i$ is open and bounded, we can assume that  
\[
\Th=(\eta_{1},\eta_{1}^{-1})\times(\eta_{2},\eta_{2}^{-1})
\]
 for $\eta_{1},\eta_{2}>0$ small enough without loss of generality. 
Then we have that 
\begin{align*}
\l|H_{p}(z;\th)\r|
&= D^{-1}\int_{0}^{z}\,\df y\int_{y}^{\infty}e^{-\b(x-y)}e^{-\g(z-y)}\,\df x \\
&=\frac{1}{\b D\g}\l(1-e^{-\g z}\r) \le\frac{1}{(c+\eta_{1}\eta_{2})\eta_{1}}\l(1-e^{-\eta_{1}^{-1} z}\r) \\
&\lesssim z, 
\end{align*}
uniformly in $\th \in \ol{\Th}$. Then, due to \eqref{vp-bdd}, we see that $|H^{f}_{\a,k}(z;\th)|\le\sqrt{2\a}|H_{p}(z;\th)|$, which yields that 
$\sup_{\th\in \ol{\Th}}\l|H^{f}_{\a,k}(z;\th)\r|\lesssim z$. 
Moreover, since it also holds that $\sup_{k\in\mathbb{N}_{0}}|\Phi_{\a,k}(x;0)|\le \sqrt{2\a}x$, we have 
\begin{align*}
\l|H^{F}_{\a,k}(z;\th)\r|
&\le\sqrt{2\a}D^{-1}\int_{0}^{z}\,\df y\int_{y}^{\infty}xe^{-\b(x-y)}e^{-\g(z-y)}\,\df x \\
&=\frac{\sqrt{2\a}}{\b D\g}\l(z-\l(\frac{1}{\g}-\frac{1}{\b}\r)(1-e^{-\g z})\r) \\
&\le\frac{\sqrt{2\a}}{(c+\eta_{1}\eta_{2})\eta_{1}}z\lesssim z.
\end{align*}
Other proofs can be done similarly, so we omit them and end the proof. 
\end{proof}

\begin{lemma}
\label{lem:gH}
For each $z\in \R_+$, $m=0, 1$ and $\k>0$, the following inequalities hold true:  
\begin{align*}
\sup_{\th\in \ol{\Th}}\l|\p_{\g}^{m}H_{p}(z;\th+\k)-\p_{\g}^{m}H_{p}(z;\th)\r|&\lesssim |\k| z,\\
\sup_{\th\in \ol{\Th}}\l|\p_{\g}^{m}H^{f}_{\a,k}(z;\th+\k)-\p_{\g}^{m}H^{f}_{\a,k}(z;\th)\r|&\lesssim |\k|z,\\
\sup_{\th\in \ol{\Th}}\l|H^{F}_{\a,k}(z;\th+\k)-H^{F}_{\a,k}(z;\th)\r|&\lesssim |\k|z,\\
\sup_{\th\in \ol{\Th}}\l|\p_{\g}H^{F}_{\a,k}(z;\th+\k)-\p_{\g}H^{F}_{\a,k}(z;\th)\r| &\lesssim |\k|(z^{2}+z).
\end{align*}
\end{lemma}

\begin{proof}
We shall show the case where $D>0$. The case where $D=0$ is similarly proved. 
As in the previous proof, we assume that  $\Th=(\eta_{1},\eta_{1}^{-1})\times(\eta_{2},\eta_{2}^{-1})$
 for $\eta_{1},\eta_{2}>0$ small enough without loss of generality, and put $\k:=(\k_{1},\k_{2})^{\top}$. 

Denote by 
\[
h_{\th}(x,y,z):=e^{-\g(x-y)} e^{-\frac{c}{D}(x-y)}e^{-\g(z-y)}
\]
Note that, it follows for $y<x<\infty$, $0<y\le z$ that 
\begin{align}
\l|e^{-\frac{c}{D+\k_{2}}(x-y)}-e^{-\frac{c}{D}(x-y)}\r| e^{-\g(z-y)} 
\ge \l|h_{\th+\k}(x,y,z)-h_{\th}(x,y,z)\r|  \label{eq:exp_ineq}
\end{align}
Using this inequality, we have 
\begin{align*}
|H_{p}(z;\th+\k)-H_{p}(z;\th)|&=\l|\int_0^{z}\,\df y\int_{y}^\infty \l(\frac{1}{D+\k_{2}}h_{\th+\k}(x,y,z)-\frac{1}{D}h_{\th}(x,y,z)\r)\,\df x\r|\\
&\le\l|\int_0^{z}\,\df y\int_{y}^\infty  \l(\frac{1}{D+\k_{2}}-\frac{1}{D}\r)h_{\th+\k}(x,y,z)\,\df x\r| \\
&\qquad +\l|\int_0^{z}\,\df y\int_{y}^\infty  \frac{1}{D}(h_{\th+\k}(x,y,z)-h_{\th}(x,y,z))\,\df x\r| \\ 
&\le\int_0^{z}\,\df y\int_{y}^\infty \l(\frac{1}{D}-\frac{1}{D+\k_{2}}\r)e^{-\frac{c}{D+\k_{2}}(x-y)}e^{-\g(z-y)}\,\df x \\
&\qquad +\int_0^{z}\,\df y\int_{y}^\infty \frac{1}{D}\l(e^{-\frac{c}{D+\k_{2}}(x-y)}-e^{-\frac{c}{D}(x-y)}\r)e^{-\g(z-y)}\,\df x\\
&=\frac{2}{c\g D}\l(1-e^{-\g z}\r)\k_{2} \\
&\le\frac{2}{c\eta_{1}\eta_{2}}\l(1-e^{-\eta_{1}^{-1} z}\r)|\k| \lesssim |\k|z.
\end{align*}
uniformly in $\th \in \ol{\Th}$. 

As for $H_{\a,k}^f$, since $\l|H^{f}_{\a,k}(z;\th+\k)-H^{f}_{\a,k}(z;\th)\r|\le\sqrt{2\a}|H_{p}(z;\th+\k)-H_{p}(z;\th)|$ by \eqref{vp-bdd}, we obtain that 
\begin{align*}
\sup_{\th\in \ol{\Th}}\l|H^{f}_{\a,k}(z;\th+\k)-H^{f}_{\a,k}(z;\th)\r|\lesssim z. 
\end{align*}
Similarly, from \eqref{vp-bdd} and \eqref{eq:exp_ineq}, we have 
\begin{align*}
\big|H^{F}_{\a,k}(z;\th+\k) &- H^{F}_{\a,k}(z;\th)\big|\\
&\le\sqrt{2\a}\l|\int_0^{z}\,\df y\int_{y}^\infty 
x\l(\frac{1}{D+\k_{2}}h_{\th+\k}(x,y,z)-\frac{1}{D}h_{\th}(x,y,z)\r)\,\df x\r|\\
&\le\sqrt{2\a}\l|\int_0^{z}\,\df y\int_{y}^\infty 
\l(\frac{1}{D+\k_{2}}-\frac{1}{D}\r)xh_{\th+\k}(x,y,z)\,\df x\r|\\
&\qquad +{\sqrt{2\a}}\l|\int_0^{z}\,\df y\int_{y}^\infty 
\frac{1}{D}(xh_{\th+\k}(x,y,z)-xh_{\th}(x,y,z))\,\df x\r|\\
&\le\sqrt{2\a}\int_0^{z}\,\df y\int_{y}^\infty 
\l(\frac{1}{D}-\frac{1}{D+\k_{2}}\r)x e^{-\frac{c}{D+\k_{2}}(x-y)}e^{-\g (z-y)}\,\df x\\
&\qquad +\frac{\sqrt{2\a}}{c}\int_0^{z}\,\df y\int_{y}^\infty 
\frac{1}{D}\l(xe^{-\frac{c}{D+\k_{2}}(x-y)}-x e^{-\frac{c}{D}(x-y)}\r) e^{-\g (z-y)}\,\df x\\
&=\frac{\sqrt{2\a}}{c\g D}\l(2z+\l(\frac{3D+2\k_{2}}{c}-\frac{2}{\g}\r)\l(1-e^{-\g z}\r)\r)\k_{2} \\
&\le\frac{\sqrt{2\a}}{c\eta_{2}}\l(\frac{2}{\eta_{1}}z+\l(\frac{5}{c\eta_{2}}-2\eta_{1}\r)\frac{1}{\eta_{1}}\l(1-e^{-\eta_{1}^{-1} z}\r)\r)|\k| \\
& \lesssim |\k|z, 
\end{align*}
uniformly in $\th \in \ol{\Th}$. 

Other proofs can be done similarly, so we omit them and end the proof. 
\end{proof}

\subsection{On functions $P$, $Q_{\a,k}$, $P^{*}$ and $Q^{*}_{\a,k}$}

\begin{lemma}
\label{lem:PQ_C1}
The following statements (a) and (b) are true: 
\begin{itemize}
\item[(a)] For each $x\in \R$, the mappings $(p,\g,D)\mapsto P(x;p,\g,D)$ and $(p,\g,D)\mapsto P^{*}(x;p,\g,D)$ are of $C^1((0,1)\times \Th)$. 
\item[(b)] For each $x\in \R$, the mappings $(p,\g,D)\mapsto Q_{\a,K}(x;p,\g,D)$ and $(p,\g,D)\mapsto Q^{*}_{\a,K}(x;p,\g,D)$ are of $C^1((0,1)\times \Th)$. 
\end{itemize}
\end{lemma}

\begin{proof}
Since the statement (a) is clear from the concrete form of $P$ and $P^*$. 
For the proof of (b), it suffices to show that the function $b\mapsto\Psi_{\a,k}(x;b)$ is of $C^1(\R)$.  
For $M>0$ large enough, it follows for $z < x$ that 
\begin{align*}
\sup_{|b|< M}\l|\p_{b}\l(be^{b(x-z)}\varphi_{\a,k}(z)\r)\r|
&=\sup_{|b|< M}\l|e^{b(x-z)}\varphi_{\a,k}(z)+b(x-z)e^{b(x-z)}\varphi_{\a,k}(z)\r|\\
&\le\sqrt{2\a}\sup_{|b|< M}(1+|b|(x-z))e^{b(x-z)} \\
&<\sqrt{2\a}(1+Mx)e^{Mx}, 
\end{align*}
which is integrable on $(0,x]$ with respect to $\df z$. Hence the Lebesgue convergence theorem yields that  $b\mapsto\Psi_{\a,k}(x;b)$ is of $C^1(\R)$.  

\end{proof}

\begin{lemma}
\label{lem:PQ_consis}
Under the assumption S1, S2 and M[2], the following (a)--(d) hold true: 
\begin{itemize}
\item[(a)] $\dis \sup_{x\in \R_+}\l|P(x;\wh{p}_{n},0,\wh{D}_{n})-P(x;p_0,0,D_0)\r|\toP 0$; 
\item[(b)] $\dis \sup_{x\in \R_+}\l|Q_{\a,k}(x;\wh{p}_{n},0,\wh{D}_{n})-Q_{\a,k}(x;p_0,0,D_0)\r|\toP 0$; 
\item[(c)] $\dis \sup_{x\in \R_+}\l|P^*(x;\wh{p}_{n},0,\wh{D}_{n})-P^*(x;p_0,0,D_0)\r|\toP 0$; 
\item[(d)] $\dis \sup_{x\in \R_+}\l|Q^*_{\a,k}(x;\wh{p}_{n},0,\wh{D}_{n})-Q^*_{\a,k}(x;p_0,0,D_0)\r|\toP 0$. 
\end{itemize}
\end{lemma}

\begin{proof}
As for (a) with $D>0$, it follows from the mean value theorem that  
\begin{align*}
\sup_{x\in\R_+}&\l|P(x;\wh{p}_{n},0,\wh{D}_{n})-P(x;p_0,0,D_0)\r| \\
&\le \l|\frac{1}{c(1-\wh{p}_{n})}-\frac{1}{c(1-p_0)}\r|+\frac{1}{c(1-p_0)}\sup_{x\in\R_+}\l|e^{-c\wh{D}_{n}^{-1}x}-e^{-cD_0^{-1}x}\r| \\
&\lesssim \l|\frac{1}{c(1-\wh{p}_{n})}-\frac{1}{c(1-p_0)}\r|+\frac{1}{c(1-p_0)}|\wh{D}_{n}-D_0| \toP 0. 
\end{align*}
The proofs for (b)--(d) are similar by using the fact \eqref{vp-bdd}. The case of $D=0$ is also similar.  
\end{proof}

\subsection{Proof of Lemma \ref{lem:fs11}} \label{proof:fs11}

Note the explicit expression in Feng and Shimizu \cite{fs13}, Proposition 4.1: 
\[
W^{(q)}(x) = \l\{
\begin{aligned}
&\frac{1}{D(1-p)(\b+ \g)}\int_{[0,x)} [e^{\g(x-z)} - e^{-\b(x-z)}]\,G_q(\df z), & (D>0) \\
&\frac{1}{c(1-p)}\int_{[0,x)} e^{\g(x-z)}\,G_q(\df z),& (D=0)
\end{aligned}\r., 
\]
By the integration-by-parts of this expression by noticing the jump at the origin, we have the consequence.

\subsection{Proof of Lemma \ref{lem:tail}} \label{proof:tail}

We will show only the case where $D>0$. The same calculation applies for $D=0$. 
According to Lemmas \ref{lem:SL} and \ref{lem:sz19}, we should show that $f:=\ol{G}_q$ satisfies \eqref{SL-cond}. 

{\red Let  $\e,M>0$ be any constants with $\e < M$.  
We firstly show that the infinite sum $\ol{G}_q(x)= \sum_{k=0}^\infty (1-p)p^k \int_x^\infty f_q^{*k}(z)\,\df z$ can be differentiable under the summation sign on $(\e,M)$. 

We shall put the partial sum $g_N(x):= \sum_{k=0}^N (1-p)p^k \int_x^\infty f_q^{*k}(z)\,\df z$ for an integer $N \in \N$, and note that $g_N(x)\to \ol{G}_q(x)\ (N\to \infty)$ for all $x>0$.  
Next, note the differential formula for the convolution: for functions $f,g \in C^1(\R_+)$, 
\[
\frac{\df }{\df x} (f*g)(x) = (f * \p_x g)(x) + f(x) g(0),  
\] 
In particular, if $f(0)=0$, then it follows for $k\in \N$ that 
\begin{align}
\p_x f^{*k}= f*\p_x f^{*(k-1)} = \dots = f^{*(k-1)}*\p_x f, \label{differentiate}
\end{align}
by induction. Using this formula, we have for $m=1,2,\dots$ that 
\begin{align*}
 \p^m_x g_N(x) =p(1-p^{m-2})\p_x^{m-1}f_q(x) - p^{m-1}\sum_{k=0}^{N-m+1} (1-p)p^k f_q^{*k}* \p^{m-1}_xf_q(x), 
\end{align*}
where we regard $\p_x^0 f_q \equiv 0$ as a convention. 
Now, $f_q(x)$ is bounded by Lemma \ref{lem:fq}, so $f_q^{*k}\ (k=1,2,\dots)$ is also bounded. 
Moreover, it also follows that $\p^m_xf_q$ is bounded on $(\e,M)$. Indeed, 
\begin{align*}
\p_x f_q(x) &= \frac{1}{pD}\l[-\b f_q(x)  + \ol{\nu}(x) + \int_x^\infty \ol{\nu}(z)\,\df z\r],  
\end{align*}
and we see by induction that, for any $m\ge 2$,  
\begin{align*}
\l|\p_x^{m-1}  f_q(x)\r| \lesssim  1+ \sum_{k=1}^{m-1} \l|\p_x^{k-1} \ol{\nu}(x)\r| + \l|\int_x^\infty \ol{\nu}(z)\,\df z\r|. 
\end{align*}
and it follows for any $\e>0$ that 
\begin{align}
\sup_{x \in (\e,M)}|\p_x^{m-1} f_q(x)| &\lesssim  1 + \sum_{k=1}^{m-1} \sup_{x \in (\e,M)} \l|\p_x^{k-1} \ol{\nu}(x)\r| + \int_M^\infty \ol{\nu}(z)\,\df z 
\lesssim  1 + M^C, \label{f-diff}
\end{align}
by the assumption \eqref{nu-bar}. 
Therefore we can confirm the sequnece $\{\p^m_x g_N(x)\}_{N\in \N}$ is a uniform Cauchy sequence on $(\e,\infty)$: 
\[
\lim_{N,N'\to \infty} \sup_{x\in (\e,M)}|\p^m_x g_N(x) - \p^m_x g_{N'}(x)| = 0, 
\]
Hence ``the term differential theorem" says that, for any $\e, M>0$, 
\[
\lim_{N\to \infty} \p^m_x g_N(x)= \p_x^m \ol{G}_q(x),\quad x \in (\e, M). 
\]
Therefore, it follows for any $x >0$ that 
\begin{align}
\p_x^m \ol{G}_q (x) &= p(1-p^{m-2})\p_x^{m-1}f_q(x) - p^{m-1} \l(\sum_{k=0}^\infty (1-p)p^k f_q^{*k} \r)*\p_x^{m-1}f_q  \notag\\
&=p(1-p^{m-2})\p_x^{m-1}f_q(x)  -p^{m-1}\,\E\l[\p_x^{m-1}f_q(x - Z)\r],   \label{Z-Gp}
\end{align}
where $Z$ is a random variable with the distribution $G_q$ whose probability density is given by $\sum_{k=0}^\infty (1-p)p^k f_q^{*k}(x)$. 

Now, as for the second term in the last right-hand side of \eqref{Z-Gp}, we see that 
\[
\E\l[\p_x^{m-1}f_q(x - Z)\r] = O\l(1 + x^\k\r),\quad x\to \infty
\]
Indeed, it follows from \eqref{f-diff} that 
\[
\sup_{x > \e} \frac{\l|\p_x^{m-1}f_q(x - Z)\r|}{1 + x^C} \lesssim \sup_{x > \e} \frac{1 + x^C + Z^C}{1 + x^C} \le 1 + |Z|^C, 
\]
and the last term is integrable by the assumption. Then, the Lebesgue convergence theorem and the equality \eqref{Z-Gp} yield that }
\begin{align*}
\p_x^m \ol{G}_q (x) & = O\l(1 + x^\k\r),\quad x\to \infty, 
\end{align*}
which implies that $x^{m/2}\p_x^m \ol{G}_q (x)  \in L^2_\a(\R_+)$ for any $\a>0$ and $m\ge 2$.  
Similarly, as \eqref{nu-bar2}, it holds that 
\[
\l|x^{m/2} \p_x^m \ol{G}_q (x)\r|^2 e^{-\a x}= O\l(x^m e^{-(\a -2k) x}\r),\quad x\to \infty, 
\]
and therefore $x^{m/2}\p_x^m \ol{G}_q (x)  \in L^2_\a(\R_+)$ for any $\a>2\k$. 
This completes the proof. 

\subsection{Proof of Theorem \ref{thm:W}} \label{proof:W}

As for the consistency, thanks to Corollary \ref{cor:sz-G} and Lemma \ref{lem:PQ_C1}, 
the continuous mapping theorem yields that 
\begin{align*}
\wh{W}^{(q)}_{K}(x)\toP W^{(q)}_{K}(x)
\end{align*}
for each $x\in\R_+$ and $q\ge 0$. In particular, as $q=0$, noticing that $\wh{\g}_{n}=\g_0=0$, we see from Lemma \ref{lem:PQ_consis} that 
\begin{align*}
\sup_{x\in\R_+}\l|\wh{W}^{(0)}_{K}(x)-W^{(0)}_{K}(x)\r|
&\leq\sup_{x\in\R_+}\l|P(x;\wh{p}_{n},0,\wh{D}_{n})-P(x;p_0,0,D_0)\r| \\
&\qquad +\sum_{k=0}^{K}\wh{a}^{G}_{\a,K}\sup_{x\in\R_+}\l|Q_{\a,k}(x;\wh{p}_{n},0,\wh{D}_{n})-Q_{\a,k}(x;p_0,0,D_0)\r|\\
&\qquad +\sum_{k=0}^{K}\l|\wh{a}^{G}_{\a,k}-a^{G}_{\a,k}\r|\sup_{x\in\R_+}\l|Q_{\a,k}(x;p_0,0,D_0)\r| 
\toP 0. 
\end{align*}

As for the asymptotic normality, we only show the case where $D>0$ since the proof for $D=0$ is similarly done. 

Note that
\begin{align*}
\sqrt{T_{n}}\l(\wh{W}^{(q)}_{K}(x)-W^{(q)}_{K}(x)\r)
&=\sqrt{T_{n}}\l(P(x;\wh{p}_{n},\wh{\g}_{n},\wh{D}_{n})-P(x;p_0,\g_0,D_0)\r) \\
&\qquad -\sum_{k=0}^{K}\wh{a}^{G}_{\a,k}\sqrt{T_{n}}\l(Q_{\a,k}(x;\wh{p}_{n},\wh{\g}_{n},\wh{D}_{n})-Q_{\a,k}(x;p_0,\g_0,D_0)\r)\\
&\qquad -\sum_{k=0}^{K}Q_{\a,k}(x;p_0,\g_0,D_0)\sqrt{T_{n}}\l(\wh{a}^{G}_{\a,k}-a^{G}_{\a,K}\r)\\
&=:U_1+U_2+U_3.
\end{align*}

On $U_1$, applying the mean value theorem, there exists some $p^{*}_{n}$, $\g^{*}_{n}$, $D^{*}_{n}$ such that 
\begin{align*}
U_1&=\sum_{k=0}^{K}\p_{(p,\g)}P(x;p^{*}_{n},\g^{*}_{n},D^{*}_{n}) \sqrt{T_{n}}
    \begin{pmatrix}
    \wh{p}_{n}-p_0\\
    \wh{\g}_{n}-\g_0
    \end{pmatrix}  \\
&\qquad +\sum_{k=0}^{K}\p_{D}P(x;p^{*}_{n},\g^{*}_{n},D^{*}_{n})\sqrt{T_{n}}\l(\wh{D}_{n}-D_0\r)\\
&= \sum_{k=0}^{K}\p_{(p,\g)}P(x;p_0,\g_0,D_0) \sqrt{T_{n}}
    \begin{pmatrix}
    \wh{p}_{n}-p_0\\
    \wh{\g}_{n}-\g_0
    \end{pmatrix} 
+ o_p(1), 
\end{align*}
by Lemma \ref{lem:D-est} and the continuous mapping theorem. 
By a similar argument as above, we have 
\begin{align*}
U_{2}=-\sum_{k=0}^{K}\wh{a}^{G}_{\a,k}\p_{(p,\g)}Q_{\a,k}(x;p_0,\g_0,D_0)\sqrt{T_{n}}
    \begin{pmatrix}
    \wh{p}_{n}-p_0\\
    \wh{\g}_{n}-\g_0
    \end{pmatrix}
    +o_{p}(1). 
\end{align*}
Moreover, on $U_3$, we obtain by \eqref{eq:aG_an} in Corollary \ref{cor:sz-G} that 
\begin{align*}
U_3=\mb{Q}_{\a,K}(x;p_0,\g_0,D_0)^\top A_{K}^{-1}\wh{B}_{K}\sqrt{T_{n}}
    \begin{pmatrix}
    \wh{\mb{a}}^{f}_{\a,K}-\mb{a}^{f}_{\a,K}\\
    \wh{\mb{a}}^{F}_{\a,K}-\mb{a}^{F}_{\a,K}
    \end{pmatrix}. 
\end{align*}
As a consequence, we have 
\begin{align}
U_1+U_2+U_3=C_K(x)\sqrt{T_{n}}\l(
    \begin{matrix}
    \wh{\mb{a}}^{f}_{\a,K}-\mb{a}^{f}_{\a,K}\\
    \wh{\mb{a}}^{F}_{\a,K}-\mb{a}^{F}_{\a,K}\\
    \wh{p}_{n}-p_0\\
    \wh{\g}_{n}-\g_0
    \end{matrix}
\r)+o_{p}(1),   \label{eq:WqK_an}
\end{align}
and Theorem \ref{thm:asymNormal} yields the consequence. 

\subsection{Proof of Theorem \ref{thm:Z}} \label{proof:Z}

By the same argument as in the proof of Theorem \ref{thm:W} in Section \ref{proof:W}, we have that 
\begin{align}
\sqrt{T_{n}}\l(\wh{Z}^{(q)}_{K}(x)-Z^{(q)}_{K}(x)\r)= qC^*_K(x)\sqrt{T_{n}}\l(
    \begin{matrix}
    \wh{\mb{a}}^{f}_{\a,K}-\mb{a}^{f}_{\a,K}\\
    \wh{\mb{a}}^{F}_{\a,K}-\mb{a}^{F}_{\a,K}\\
    \wh{p}_{n}-p_0\\
    \wh{\g}_{n}-\g_0
    \end{matrix}
\r)+o_{p}(1),   \label{eq:ZqK_an}
\end{align}
and Theorem \ref{thm:asymNormal} yields the consequence. 

\subsection{Proof of Theorem \ref{cor:WqK_ZqK_an}} \label{proof:WqK_ZqK_an}
Noticing the expressions \eqref{eq:WqK_an} and  \eqref{eq:ZqK_an}, 
we have 
\begin{align*}
\sqrt{T_{n}}\l(
    \begin{matrix}
    \wh{W}^{(q)}_{K}(x)-W^{(q)}_{K}(x)\\
    \wh{Z}^{(q)}_{K}(x)-Z^{(q)}_{K}(x)
    \end{matrix}
\r) = \sqrt{T_n} \begin{pmatrix}
    C_{K}(x)\\
    qC^{*}_{K}(x)
    \end{pmatrix} 
 \begin{pmatrix}
    \wh{\mb{a}}^{f}_{\a,K}-\mb{a}^{f}_{\a,K}\\
    \wh{\mb{a}}^{F}_{\a,K}-\mb{a}^{F}_{\a,K}\\
    \wh{p}_{n}-p_0\\
    \wh{\g}_{n}-\g_0
    \end{pmatrix}
    + o_p(1). 
\end{align*}
Then Theorem \ref{thm:asymNormal} yields the consequence. 


\end{document}